\documentclass[final]{siamart220329}
\usepackage{amsmath,latexsym,amssymb,amsfonts,mathrsfs}
\usepackage{epsf,graphicx,epsfig,color,cite,cases}
\usepackage{subfigure,graphics,multirow,marginnote,graphicx}
\usepackage{enumerate}
\usepackage{dsfont}
\usepackage{epstopdf}
\usepackage{bm}

\usepackage{algorithm}
\usepackage{todonotes}
\usepackage{float}

\usepackage{verbatim}
\usepackage{pgfplots}
\pgfplotsset{compat=newest}
\usetikzlibrary{plotmarks}
\usetikzlibrary{arrows.meta}
\usepgfplotslibrary{patchplots}
\usepackage{grffile}

\newcommand{\near}{}

\newcommand{\no}{\nonumber}
\newcommand{\be}{\begin{eqnarray}}
\newcommand{\ben}{\begin{eqnarray*}}
\newcommand{\en}{\end{eqnarray}}
\newcommand{\enn}{\end{eqnarray*}}

\newcommand{\dx}{\,\mathrm{d}x}
\newcommand{\dy}{\,\mathrm{d}y}

\newcommand{\ds}{\,\mathrm{d}s}

\newcommand{\us}{u^{\mathrm{s}}}
\newcommand{\ui}{u^{\mathrm{i}}}

\newcommand{\Green}{G_{\!D}}
\newcommand{\GreenD}[1]{G_{#1}}
\newcommand{\Nop}{\mathcal{G}_{D}}
\newcommand{\NopD}[1]{\mathcal{G}_{#1}}
\DeclareMathOperator{\Diag}{Diag}
\DeclareMathOperator{\HS}{HS}
\newcommand{\Nsample}{N_{\mathrm{sample}}}
\newcommand{\Nmeas}{N_{\mathrm{meas}}}

\newcommand{\Nsrc}{N_{\mathrm{src}}}
\DeclareMathOperator{\supp}{supp}
\newcommand{\shapeman}{\mathcal{M}}
\newcommand{\shape}{\rho}
\newcommand{\paren}[1]{\left(#1\right)}
\DeclareMathOperator{\EV}{\mathbf{E}}
\DeclareMathOperator{\Cov}{\mathbf{Cov}}

\newsiamthm{thm}{Theorem}
\newsiamthm{prop}{Proposition}
\newsiamthm{lem}{Lemma}

\newsiamthm{assumption}{Assumption}
\newsiamremark{problem}{Problem}
\newsiamremark{example}{Example}
\newsiamthm{defi}{Definition}
\newsiamthm{thmdefi}{Theorem and Definition}

\newsiamremark{remark}{Remark}

\renewcommand{\theequation}{\arabic{section}.\arabic{equation}}

\usepackage{lipsum}
\usepackage{amsfonts}
\usepackage{graphicx}
\usepackage{epstopdf}
\usepackage{algorithmic}
\ifpdf
  \DeclareGraphicsExtensions{.eps,.pdf,.png,.jpg}
\else
  \DeclareGraphicsExtensions{.eps}
\fi

\newsiamremark{hypothesis}{Hypothesis}
\crefname{hypothesis}{Hypothesis}{Hypotheses}
\newsiamthm{claim}{Claim}

\headers{Passive inverse obstacle scattering problems for the Helmholtz equation}{T. Hohage and M. Liu}

\title{Passive inverse obstacle scattering problems for the Helmholtz equation
}
\author{Thorsten Hohage\thanks{Institute for Numerical and Applied Mathematics, University of Göttingen, Germany, and Max Planck Institute for Solar System Research, Germany
(\email{hohage@math.uni-goettingen.de}).}
\and Meng Liu\thanks{Max Planck Institute for Solar System Research, Germany
(\email{liumeng@mps.mpg.de}).}}

\usepackage{amsopn}
\DeclareMathOperator{\diag}{diag}

\begin{document}
\maketitle

\begin{abstract}
Passive imaging involves recording waves generated by uncontrolled, random sources and utilizing correlations of such waves to image the medium through which they propagate. In this paper, we focus on passive inverse obstacle scattering problems governed by the Helmholtz equation in $\mathbb{R}^d\;(d=2,3)$. The random source is modelled by a Gaussian random process. 
Uniqueness results are established for the inverse problems to  determine the source strength or shape and location of an obstacle, or both of them simultaneously from near-field correlation measurements. Finally, we present efficient methods for numerical reconstructions.
\end{abstract}

\begin{keywords}
passive inverse obstacle scattering, correlation data, inverse source, inverse obstacle, inverse source-obstacle, uniqueness
\end{keywords}


\pagestyle{myheadings}
\thispagestyle{plain}
\markboth{}{Passive inverse obstacle scattering problems for the Helmholtz equation}



\section{Introduction}
An active new trend in inverse scattering theory is passive imaging, which has applications in numerous fields such as seismology of the Earth and the Sun, oceanography, non-destructive testing, aero-acoustics, or incoherent X-ray diffraction imaging. Roughly speaking, passive imaging consists in extracting information from what would classically be considered as noise, more precisely from correlations of randomly excited waves. 

We study the scattering of randomly excited 
time-harmonic acoustic or transverse electric waves from a bounded obstacle 
$D\subset \mathbb{R}^d\;(d=2,3)$ with smooth 
boundary $\partial D\in\mathcal{C}^2$ and connected complement $\mathbb{R}^d\setminus \overline{D}$. 
The source will be modelled 
by a Gaussian random process $g$ with mean zero and support in a bounded domain  $\Omega\subset\mathbb{R}^d\backslash\overline{D}$. Moreover, let $M$ denote 
some closed Lipschitz measurement curve or surface $M$ containing $\overline{D}$ and $\overline{\Omega}$ in its interior. The geometrical setup is sketched in Fig.~\ref{fig:geometry}.

Without the obstacle $D$, the source $g$ would generate a field $\ui:\mathbb{R}^d\to \mathbb{C}$ 
considered as an incident field, which is uniquely described by 
\be
\label{ui}
\begin{cases}
-\Delta\ui-\kappa^2\ui=g&{\rm in\;}\mathbb{R}^d,\\
\displaystyle{\lim_{r\to\infty}}r^{\frac{d-1}{2}}\left(\frac{\partial \ui}{\partial r}-i\kappa \ui\right)=0&{\rm for\;} r=|x|.
\end{cases}
\en
Here $\kappa>0$ is the wave number. 
The obstacle $D$, which is assumed to be sound-soft, then gives rise to a scattered 
field $\us$, which is superimposed to $\ui$ to 
a total field $u = \ui+\us$ and satisfies
\be
\label{us}
\begin{cases}
\Delta\us+\kappa^2\us=0 & {\rm in\;} \mathbb{R}^d\backslash\overline{D},\\
\us=-\ui & {\rm on\;}\partial D,\\
\displaystyle{\lim_{r\to\infty}}r^{\frac{d-1}{2}}\left(\frac{\partial \us}{\partial r}-i\kappa \us\right)=0&{\rm for\;} r=|x|.
\end{cases}
\en
The last condition in \eqref{us}, known as the Sommerfeld radiation condition.
\begin{figure}[htbp]
\centering
\input{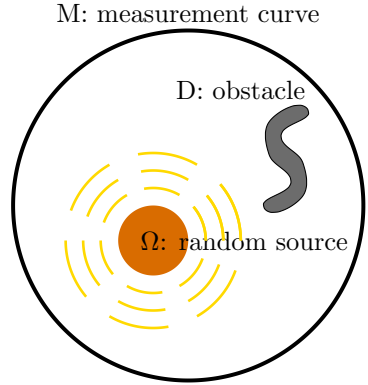}
\caption{\label{fig:geometry}Sketch of the geometrical setup.}
\end{figure}

The inverse problems can be formulated as follows:
\begin{problem}\label{problem}
From near-field correlation measurements of 
\ben
c_{D,{\rm Cov}[g]}\near(x,y):={\rm Cov}(u(x),u(y)),\;x,y\in M
\enn
with $u$ described by \eqref{ui} and \eqref{us}, determine $\Cov[g]$ and/or $D$.
\end{problem}

It will be natural and convenient to work with the covariance operator 
$\mathcal{C}(D,\Cov[g]): L^2(M)\to L^2(M)$ of the random vector $u|_M$, which is the integral operator with kernel $c_{D,{\rm Cov}[g]}$:
\[
[\mathcal{C}(D,\Cov[g])f](x) = \int_M c_{D,\mathrm{Cov}[g]}\near(x,y)f(y)\ds(y),\quad 
x\in M.
\]
Our main theoretical result shows that under certain
assumptions noise-free data $c_{D,\mathrm{Cov}[g]}\near(x,y)$ for a fixed known wave number $\kappa\in\mathbb{R}_{+}$ uniquely determine the obstacle $D$, and for uncorrelated sources $g$ also 
$\Cov[g]$ as well as both of them together. 

\begin{thm}
\label{sourceshape}
Let $D_1, D_2\subset \mathbb{R}^d$ be open, bounded sets such 
that $\mathbb{R}^d\setminus D_i$ are connected for $i\in\{1,2\}$. 
Let $\Omega\subset\mathbb{R}^d$ be another open, bounded set such that 
$\overline{\Omega}\cap(\overline{D_1\cup D_2})=\emptyset$. 
Moreover, let $M$ be the boundary of some bounded domain containing 
$\overline{\Omega\cup D_1\cup D_2}$.  
Denote by $\Cov[g_i]: L^2(\Omega_i)\rightarrow L^2(\Omega_i)$ the covariance operators corresponding to two different random sources $g_i$ supported on open domains $\Omega_i\subset\Omega$. Then we have the following results:

\begin{itemize}	
\item [(a)] If the operators $\Cov[g_i]$ are positive definite (but not necessarily 
diagonal, i.e., the source processes $g_i$ may be spatially correlated), then $\mathcal{C}\near(D_1,\Cov[g_1])=\mathcal{C}\near(D_2,\Cov[g_2])$ implies $D_1=D_2$;
\item [(b)] If the source processes $g_i$ are spatially uncorrelated with covariance operators given by multiplication operators $\Cov[g_i]=M_{q_i}$ 
and source strength functions $q_i\in L^{\infty}(\Omega)$, $q_i>0$, then $\mathcal{C}\near(D_1,M_{q_1})=\mathcal{C}\near(D_2,M_{q_2})$ implies $D_1=D_2$ and $q_1=q_2$.
\end{itemize}
\end{thm}

In the field of inverse scattering problems, extensive literature addresses the inverse source, inverse obstacle, inverse potential, and inverse medium problems related to wave equations, using either passive or active measurements. It is well-established that classical inverse source problems do not yield unique solutions due to the presence of nonradiating sources, as discussed in works such as \cite{Bleistein1977} and \cite{Devaney1982}. Possibilities to face this non-uniqueness include imposing additional constraints \cite{Devaney2007}, retrieving partial information from sources \cite{GS:17}, and using multi-frequency data \cite{Bao2010}. We do not attempt to review the rich literature on inverse obstacle scattering problems here, but refer to the monograph \cite{Colton2019}, \cite{Kirsch1993} for uniqueness results, and \cite{Kirsch1993(2)} for the Fr\'echet derivative used in iterative inversion methods.

Regarding random inverse scattering problems, there are established results on inverse random sources and random potential problems. The uniqueness of the source strength given far field 
correlation data is discussed in \cite{Devaney1979}.
Numerical reconstructions of both expectation 
pointwise variance of a random source from multifrequency data
are presented in \cite{Li2018}.
For the uniqueness of the principal symbol of the covariance operator in random potential problems with scattering data in a frequency band, see \cite{Lassas2008, Caro2019,Li2022(2)}. Simultaneous recovery of random sources and both deterministic and random potentials from a single realization of the stochastic 
process and many frequencies 
is investigated in \cite{Li2019} and \cite{Li2021}. Furthermore, the linear sampling method 
has been extended 
specifically for passive obstacle scattering problems in 
\cite{Garnier2023}, focusing on sound-soft inverse acoustic scattering with randomly distributed point sources.

Passive imaging has also been applied in various fields such as helioseismology \cite{Duvall1993,Mueller2023} and aeroacoustics \cite{Hohage2020}. Numerous additional references can be found in the monograph \cite{Garnier2016}. Uniqueness results in helioseismology are presented in \cite{Agaltsov2018, Agaltsov2020}.

Compared to these existing results, the main novelty of 
this paper is the first proof of unique identifiability 
of the shape of an obstacle in a passive imaging setting. 

The remainder of this paper is organized as follows: In Section 2, we formalize the definition of the forward problem and show its well-posedness. Section 3 contains the proof of 
Theorem \ref{sourceshape}. We proceed by analyzing the 
Fr\'echet derivatives of the forward operators and their adjoints in Section 4, before presenting 
reconstruction methods and numerical examples in Section 5. 
The paper ends with a section presenting the final conclusions, followed by an appendix that summarizes results on Hilbert-Schmidt and trace-class operators, along with Gaussian processes and their associated covariance operators.

\section{The forward problem}
\subsection{Well-posedness of the obstacle scattering problem for deterministic distribution-valued sources}

In this subsection we study the well-posedness of the obstacle scattering problem \eqref{ui}-\eqref{us} and the regularity properties of solutions $\ui$ and $\us$. Some challenges arise from low path-wise
regularity of many random source processes $g$, in particular if 
we assume $g$ to be spatially uncorrelated. In this subsection we 
treat $g$ as a deterministic distribution $g\in C^{\infty}(\Omega)'$.


For the random source problem described in \eqref{ui} in $\mathbb{R}^d\;(d=2,3)$, the existence and uniqueness of solutions has been established in certain function spaces when the source term is driven by additive colored noise, see \cite{Bao2016}. Additionally, \cite{Li2020} explores the case where the source is modeled as a microlocally isotropic generalized Gaussian field with compact support.


Recall that the fundamental solution to 
the Helmholtz equation satisfying the Sommerfeld radiation 
condition is given by
\ben
\Phi\left(x\right): =
\begin{cases}
\dfrac{i}{4}H^{(1)}_0\left(\kappa\left| x\right|\right)&\quad x\in\mathbb{R}^2\backslash\{0\},\\
\dfrac{{\exp(i\kappa\left|x\right|})}{4\pi\left|x\right|}&\quad x\in\mathbb{R}^3\backslash\{0\},
\end{cases}
\enn
where $H^{(1)}_0$ denotes the Hankel function of the first kind of order $0$. 

For smooth incident fields $\ui$ well-posedness of the scattering 
problem \eqref{us} in natural function spaces can be established by variational methods (see \cite{Hohage1999}), and 
in the case of Dirichlet boundary conditions no regularity assumptions 
on $D$ are required. 

In particular, we can choose point sources $\ui = \Phi(\cdot,y)$ with $y\in \mathbb{R}^d\setminus \overline{D}$ as incident fields. Denoting the corresponding scattered 
fields by $G^{\mathrm{s}}_{D}(\cdot,y):=\us$ with $\us$ 
given by \eqref{us} for $\ui$ as above, 
the corresponding total fields yield the 
Green function  
$\Green(x,y)=\Phi(x,y)+\Green^{\mathrm{s}}(x,y)$ 
of the scattering problem \eqref{us}. It can be
uniquely characterized by
\be
\label{2.3}
\begin{cases}
\Delta \Green(\cdot,y)+\kappa^2\Green(\cdot,y)=-\delta_{y}& {\rm in\;} \mathbb{R}^d\backslash\overline{D},\\
\Green(\cdot,y)=0& {\rm on\;}\partial D,\\
\displaystyle{\lim_{r\to\infty}}r^{\frac{d-1}{2}}\left(\frac{\partial \Green(\cdot,y)}{\partial r}-i\kappa \Green(\cdot,y)\right)=0&{\rm for\;} r=|x|
\end{cases}
\en
for all $y\in \mathbb{R}^d\backslash\overline{D}$.
\begin{prop}
\label{solution}
Let $\Omega\subset \mathbb{R}^d$ be bounded. 
Then for any deterministic distribution $g\in 
C^{\infty}(\Omega)'$ with $\supp\; g\subset \Omega$, Problem
\eqref{ui} has a unique distributional solution 
$\ui\in C^{\infty}_c(\mathbb{R}^d)'$
given by 
\be
\label{ui_1}
\ui(x)=\int_{\Omega}\Phi(x-y)g(y)\dy, \quad x\in\mathbb{R}^d
\en
in the sense of a convolution with the distribution 
$\Phi\in C^{\infty}_c(\mathbb{R}^d)'$. 
We have $\ui\in C^{\infty}(\mathbb{R}^d\setminus \overline{\Omega})$. Moreover, if $g\in H^{s}(\mathbb{R}^d)$, then 
$\ui\in H^{s+2}_{\mathrm{loc}}(\mathbb{R}^d)$.

The solution of problem \eqref{us} is given by 
\be
\label{us_1}
u(x) = \int_{\Omega}\Green(x,y)g(y)\dy,
\quad x\in\mathbb{R}^d\setminus \overline{D}.
\en
We have $u\in C^{\infty}(\mathbb{R}^d\setminus (\overline{\Omega}\cup\overline{D}))$, and if $g\in H^{s}(\mathbb{R}^d)$, then 
$u \in H^{s+2}_{\mathrm{loc}}(\mathbb{R}^d\setminus \overline{D})$.
\end{prop}

\begin{proof}
It is well known that \eqref{ui_1} defines a distributional solution to \eqref{ui}. 
As $\Phi$ and all its derivatives are uniformly bounded on compact subsets 
of $\mathbb{R}^d\setminus\{0\}$, it follows that $\ui$ is infinitely smooth on $\mathbb{R}^d\setminus \overline{\Omega}$. As the values $\ui$ on some neighborhood $\tilde{\Omega}$ of $\Omega$ only depend on $\Phi$ restricted to the 
ball $B(R)$ of radius $R>\mathrm{diam}\;\Omega$ around $0$, 
we 
may replace \eqref{ui_1} by a periodic convolution 
on a cube with side length $>2R$
with a function $\Phi_{\mathrm{per}}$ which coincides 
with $\Phi$ on $B(R)$ 
and smoothly decays to $0$ at the boundaries of the 
periodicity cell (see \cite[\S 3.1]{vainikko:01} for details). 
The Fourier coefficients of $\Phi_{\mathrm{per}}$ can 
be shown to decay of order $|\widehat{\Phi}_{\mathrm{per}}(j)|\leq C|j|^{-2}$ for $
j\neq 0$, and this implies that $\ui|_{\tilde{\Omega}}\in H^{s+2}(\tilde{\Omega})$. 

Given sufficient regularity of $\ui$, 
uniqueness follows from a classical energy flux argument: If $g=0$, then by Green's formula $\Im \int_{\mathbb{S}^{d-1}} \ui \frac{\partial\overline{\ui}}{\partial\nu}\,\mathrm{d}s=0$, and hence $\ui\equiv 0$ due 
to \cite[Thm.~2.12]{Colton2019}. 
But from the identity $(-\Delta + I)\ui = (\kappa^2+1)\ui$ one can see that 
$\ui\in H^t_{\mathrm{loc}}(\mathbb{R}^d)$ implies $\ui\in H^{t+2}_{\mathrm{loc}}(\mathbb{R}^d)$ for all 
$t\in\mathbb{R}$, so $\ui\in C^{\infty}(\mathbb{R}^d)$. 

We claim that $\Green^{\mathrm{s}}\in C^{\infty}((\mathbb{R}^d\setminus \overline{D})^2)$. 
By elliptic regularity we have $\Green^{\mathrm{s}}(\cdot,y)\in C^{\infty}(\mathbb{R}^d\setminus \overline{D})$ for any $y\in \mathbb{R}^d\setminus \overline{D}$, and by the reciprocity principle 
$\Green^{\mathrm{s}}(x,y)= \Green^{\mathrm{s}}(y,x)$, 
we also have 
$\Green^{\mathrm{s}}(\cdot,x)\in C^{\infty}(\mathbb{R}^d\setminus \overline{D})$ for any $x\in \mathbb{R}^d\setminus \overline{D}$. 
By considering solution $\us$ in \eqref{ui}, \eqref{us} with 
right hand side $g$ given by derivatives of delta-distributions, it can also be shown that mixed derivatives of $\Green^{\mathrm{s}}$ of arbitrary order exist. This shows that 
\[
\us(x) := \int_{\Omega}\Green^{\mathrm{s}}(x,y) g(y)\dy,
\quad x\in \mathbb{R}^d\setminus \overline{D}
\]
is well defined and infinitely smooth. Moreover, 
by the properties of $\Green^{\mathrm{s}}$ it satisfies 
\eqref{us} for $\ui$ given by \eqref{ui_1}. 
As $u=\ui+\us$, we obtain the stated regularity properties. 
\end{proof}


\subsection{The random source process}\label{section3}

It is well known that for any positive semidefinite operator $C_g\in L(X)$
on a Hilbert space $X$ there exists a centered Gaussian process $g$ on $X$ with 
$\Cov(g)=C_g$, i.e., $\EV[\overline{\langle g,x_1\rangle}\langle g,x_2\rangle]=\langle C_g x_1,x_2\rangle$ for all $x_1,x_2\in X$ (see \cite[Prop. 2.1.10]{GN:15}). In particular, for 
any $q\in L^{\infty}(\Omega)$ with $q\geq 0$ there exists a centered 
Gaussian process $g$ on $L^2(\Omega)$ with $\Cov(g)=M_q$. 
This process can be constructed explicitly as $g=M_{\sqrt{q}}W$
from Gaussian white noise $W$ on $L^2(\Omega)$ since 
$\Cov(g)=M_{\sqrt{q}}\Cov(W)M_{\sqrt{q}}^* = M_q$ 
(see  Eq.~\eqref{eq:linOp_RP}). 

The embedding operator $J:H^k_0(\Omega)\hookrightarrow L^2(\Omega)$, $k\in \mathbb{N}$ is Hilbert-Schmidt if 
$k>\frac{d}{2}$ (see \cite[Thm.~4]{maurin:61}). It follows 
from Proposition \ref{prop:regularity_white_noise} and 
Lemma \ref{lem:finite_square_norm} that there exists a 
strict version of 
Gaussian white noise $W$ on $L^2(\Omega)$ which is  
a random variable with values in $H^{-k}(\Omega)=H^k_0(\Omega)'$, and 
\[
\EV \|W\|_{H^{-k}}^2=\|J\|^2_{\HS} <\infty
\]
(see also \cite[Thm. 4.4.3]{GN:15} for sharp pathwise regularity of results for white noise in terms of Besov spaces for $d=1$). 
Now suppose that 
\begin{align}\label{eq:q_regularity}
\sqrt{q}\in C^k(\overline{\Omega}).
\end{align}
Then $M_{\sqrt{q}}:H^k_0(\Omega)\to H^k_0(\Omega)$ is 
well-defined and bounded, and hence the adjoint 
$M_{\sqrt{q}}:H^{-k}(\Omega)\to H^{-k}(\Omega)$ is well-defined and bounded as well. This implies that there 
is a strict version of $g=\sqrt{q}W$, which is a 
random variable in $H^{-k}(\Omega)$. 
Therefore, by Proposition \ref{solution} $\ui$ and $u$ defined by \eqref{ui_1} 
and \eqref{us_1} are distributional solutions to \eqref{ui} and \eqref{us}, respectively,  with regularity 
$\ui,u\in H^{-k+2}_{\mathrm{loc}}(\mathbb{R}^d)$. 

Define the near-field operator $\Nop : L^{2}(\Omega)\rightarrow L^2(M)$
\ben
(\Nop\phi)(x):=\int_{\Omega}\Green(x,y)\phi(y){\rm d}y,\;x\in M,
\enn
where $\Green(x,y)$ is the Dirichlet Green's function of Problem \eqref{2.3} with $\Green^s(x,y)=-\Phi(x,y)$ on $\partial D$ for $y\in\Omega$. Then, according to 
\eqref{us_1} we can write 
\[
u|_M =  \Nop g.
\]
As $\operatorname{dist}(\Omega,M)>0$, $\Green$ restricted to 
$\overline{\Omega}\times M$ is continuous, and in particular 
square integrable. Therefore, $\Nop\in \HS(L^2(\Omega), L^2(M))$ (see \eqref{eq:HSL2}). Using  
Lemma \ref{lem:cov_fct_op} and Eq.~\eqref{eq:linOp_RP}
we can derive the expression 
\be
\label{3.2}
\mathcal{C}\near(D,{\rm Cov}[g])=\Cov(u|_M)
=\Cov(\Nop g)
= \Nop \Cov[g]\Nop^* 
= \Nop M_q \Nop^*
\en
for our forward operator and conclude in particular that 
$u|_M$ is a random variable in $L^2(M)$ (see Proposition 
\ref{prop:regularity_white_noise} and Lemma \ref{lem:finite_square_norm}). 
This formula also makes sense if we relax the regularity condition \eqref{eq:q_regularity} to 
\be\label{eq:relaxed_q_smoothness}
q\in L^{\infty}(\Omega),
\en
which we will assume in the following. In this case, 
we can still define random processes $\ui$ and $u$ by
\eqref{ui_1} and \eqref{us_1}, respectively, which 
are random variables in $H^l(K)$ for any smooth compact subset 
$K\subset \mathbb{R}^d\setminus \overline{\Omega}$ and 
$l\in\mathbb{N}_0$, and solutions to the Helmholtz equation. 
However, under the relaxed smoothness assumption \eqref{eq:relaxed_q_smoothness} it 
is no longer obvious in which sense $\ui$ is a solution 
to \eqref{ui} in $\Omega$.

\section{Uniqueness results}\label{uniqueness}
This section presents the unique recovery results for determining the source strength, the shape and location of an obstacle, and both simultaneously, using near-field correlation measurements.

\subsection{Identification of the source strength \texorpdfstring{$q$}{\textit{q}}}
In this subsection, we will consider the case that 
the obstacle $D$ is known, and we will show, as 
a special case of Theorem \ref{sourceshape},  
that the strength $q$ of an uncorrelated source $g$ can be uniquely determined by noise-free 
correlation data. 

\begin{lem}
\label{lem1}
The inclusion
\ben
\mathcal{W}\subset\overline{{\rm ran}(\Nop^{\ast})}\textbf{}.
\enn
holds true for the set
\ben
\mathcal{W}:=\{w\in C^{\infty}(\Omega):\Delta w+\kappa^2 w=0,\;\exists\;\phi_{w}\in C^{\infty}_{0}(\mathbb{R}^d\backslash\overline{D})\text{ s.t. }\phi_{w}|_{\partial D}=0,\;\phi_{w}|_{\Omega}=w\}.
\enn
\end{lem}

\begin{proof}
As $\operatorname{ker}(\Nop)^{\perp} = \overline{{\rm ran}(\Nop^{\ast})}$, 
it suffices to prove that $\mathcal{W}\perp \operatorname{ker}(\Nop)$. Let $w\in\mathcal{W}$ and 
$v\in\operatorname{ker}(\Nop)$. Then $u_{v}(x):=\int_{\Omega}\Green(x,y)v(y){\rm d}y$ 
with the Dirichlet Green's function $\Green$ of Problem \eqref{2.3}
vanishes on $M$. From the uniqueness of Problem \eqref{2.3} and analytic continuation, we obtain that $u_{v}=0$ in 
$\mathbb{R}^d\backslash(\overline{\Omega}\cup \overline{D})$. Together with $\Delta \phi_w+\kappa^2\phi_w=0$ in 
$\Omega$ and the compactness of $\supp \phi_w$
this shows that 
\ben
0&=&\int_{\mathbb{R}^d\backslash\overline{D}}u_v(x)\overline{(\Delta \phi_w(x)+\kappa^2\phi_w(x))}\dx\\
&=&\int_{\Omega}v(y)\left\{\int_{\mathbb{R}^d\backslash\overline{D}}\Green(x,y)\overline{(\Delta \phi_w(x)+\kappa^2\phi_w(x))}\dx\right\}\dy\,.
\enn
Combing with the Green's formula in \cite{Colton2019} and the homogeneous Dirichlet boundary condition of $\Green(x,y)$ on $\partial D$, we see that for any test function 
$\phi\in V:=\{v\in C_{0}^{\infty}(\mathbb{R}^d\backslash\overline{D}):v|_{\partial D}=0\}$ and any $y\in\mathbb{R}^d\backslash\overline{D}$, we have
\ben
\int_{\mathbb{R}^d\backslash\overline{D}}\Green(x,y)(\Delta\phi(x)+\kappa^2\phi(x))\dx=-\phi(y),
\enn
Together with the previous equation this implies that 
$0=-\int_{\Omega}v\overline{w}\dy$,
i.e.\ $v\perp w$.
\end{proof}

Motivated by \cite{Raumer2021} and \cite{Hohage2020}, we first show the following special case of Theorem \ref{sourceshape}:

\begin{thm}
\label{source}
Suppose the assumptions of Theorem \ref{sourceshape} hold true 
with $D_1=D_2=D$. 
Then $\mathcal{C}\near(D,M_{q_1})= \mathcal{C}\near(D,M_{q_2})$ implies $q_1 = q_2$.
\end{thm}

\begin{proof}
Due to the linearity of $\mathcal{C}\near(D,M_q)$ in $q$ it suffices to show that $\mathcal{C}\near(D,M_q)\equiv 0$ for some 
$q\in L^{\infty}(\Omega)$ implies $q=0$. (We will not use non-negativity of the $q_i$.) In this case 
we obtain that
\ben
0=\langle \Nop M_q\Nop^{*}v_1,v_2\rangle_{L^2(M)}=\langle M_q\Nop^{*}v_1,\Nop^{*}v_2\rangle_{L^2(\Omega)}\;\text{ for all }v_1, v_2\in L^2(M).
\enn
Therefore,
\be
\label{4.4(1)}
\langle qu_1,u_2\rangle_{L^2(\Omega)}=0\;\text{ for all }u_1, u_2\in{\rm ran}(\Nop^{*}).
\en
Since $q\in L^{\infty}(\Omega)$, it follows that \eqref{4.4(1)} also holds true for elements $u_1$, $u_2$ of the $L^2$-closure $\overline{{\rm ran}(\Nop^{*})}$, and by Lemma \ref{lem1} in particular for 
elements of $\mathcal{W}$. 
We show that 
\be\label{eq:ux_in_W}
u_{\hat{x}}\in\mathcal{W}\quad \text{for }\quad 
u_{\hat{x}}(y):=\exp(i\kappa\hat{x}\cdot y) \text { with }\hat{x}\in\mathbb{S}^{d-1}.
\en
As $\overline{\Omega}$ and $\overline{D}$ are compact and disjoint 
by assumption, it follows that $\epsilon:=\operatorname{dist}(\overline{\Omega},\overline{D})>0$. Therefore,
\[
B_{\epsilon}(\Omega):=\{x\in\mathbb{R}^{d}\;|\;\operatorname{dist}(\Omega,x)<\epsilon\}\subset\mathbb{R}^{d}\backslash\overline{D}.
\]


It follows that 
there exists a cut-off function $\eta\in C^{\infty}(\mathbb{R}^{d})$ with $\eta\equiv 1$ in $\overline{\Omega}$ and $\eta\equiv 0$ in $\mathbb{R}^{d}\backslash\overline{B_{\epsilon}(\Omega)}$. Define $\phi_{u_{\hat{x}}}:=\eta u_{\hat{x}}$ for any function $u_{\hat{x}}\in C^{\infty}(\mathbb{R}^{d}\backslash\overline{D})$, then $\phi_{u_{\hat{x}}}$ coincides with $u_{\hat{x}}$ in $\Omega$ and $\phi_{u_{\hat{x}}}=0$ on $\partial D$. 
As $u_{\hat{x}}$ satisfies the Helmholtz equation in $\Omega$, we have proved \eqref{eq:ux_in_W}. 

We have shown that we can choose $u_1$ and $u_2$ 
in \eqref{4.4(1)} as plane wave functions  $u_{\widehat{x}_1}$ and $u_{\widehat{x}_2}$ with $\widehat{x}_1, \widehat{x}_2\in \mathbb{S}^{d-1}$ to obtain
\be
\label{4.5(1)}
0=\langle qu_{\widehat{x}_1},u_{\widehat{x}_2}\rangle_{L^2(\Omega)}=\int_{\Omega}q(y)\exp(-i\kappa(\widehat{x}_2-\widehat{x}_1)\cdot y){\rm d}y.
\en
As 
\ben
\left\{\widehat{x}_2-\widehat{x}_1:\widehat{x}_1, \widehat{x}_2\in \mathbb{S}^{d-1}\right\}=\left\{\xi\in\mathbb{R}^d:|\xi|_2\leq 2\right\}
\enn
the Fourier transform of $q$, extended by zero to the whole space, \eqref{4.5(1)} satisfies 
\ben
\widehat{q}(\xi)=\int_{\mathbb{R}^d}q(y)\exp(-i\xi\cdot y){\rm d}y=0\quad \text{ for }|\xi|_2\leq 2\kappa\,. 
\enn
Since $q$ has compact support, $\widehat{q}$ is real analytic, and hence it must vanish everywhere by analytic continuation. In view of the injectivity of the Fourier transform, we obtain that $q=0$, which completes the proof.
\end{proof}

\subsection{Simultaneous identification of a source strength 
\texorpdfstring{$q$}{q} and an obstacle \texorpdfstring{$D$}{D}}\label{subsource-shape}




We are now in a position to prove our main theorem. 

\begin{proof}[Proof of Theorem \ref{sourceshape}.]
Firstly, we suppose that $D_1\neq D_2$ and $\Cov[g_1]\neq \Cov[g_2]$. 
Since the operators $\mathcal{C}\near(D_1,\Cov[g_1]) = \mathcal{C}\near(D_2,\Cov[g_2])$ 
are compact, self-adjoint and positive semidefinite and since $L^2(M)$ is separable, 
there exists a complete orthonormal system $E=\{f_j:j\in\mathbb{N}\}$ of $L^2(M)$ consisting of eigenvectors with corresponding 
eigenvalues $\lambda_j\geq 0$, i.e., 
$\mathcal{C}\near(D_i,\Cov[g_i])f_j=\lambda_jf_j$ for $j\in\mathbb{N}$ and $i\in\{1,2\}$.
   
We now prove that $\mathcal{C}\near(D_i,\Cov[g_i])$ is injective such that $\lambda_j>0$ for 
all $j\in\mathbb{N}$. Let $\mathcal{C}\near(D_i,\Cov[g_i])\psi=0$ for $\psi\in L^2(M)$. As
\be
\label{4.22}
\mathcal{C}\near(D_i,\Cov[g_i])=\NopD{D_i}\Cov[g_i]\NopD{D_i}^*,\;i=1,2,
\en
we have
\ben
0 =
\langle \NopD{D_i}\Cov[g_i]\NopD{D_i}^*\psi,\psi \rangle_{L^2(M)}=\langle \Cov[g_i]\NopD{D_i}^*\psi,\NopD{D_i}^*\psi \rangle_{L^2(\Omega)}.
\enn
From the coercivity of
$\Cov[g_i]$ on $L^2(\Omega)$, we get that $w_i:=\NopD{D_i}^*\psi$, vanishes 
on $\Omega$, i.e. 
\ben
0=\overline{w_i(x)}=\int_{M}\GreenD{D_i}(x,y)\overline{\psi(y)}{\rm d}s(y),\;\;x\in\Omega,
\enn
where $\GreenD{D_i}$ is the Dirichlet Green function with respect to the obstacle $D_i$. Since we have assumed that $\supp q_i$ contains an open set, 
by analytic continuation 
and the uniqueness of exterior Dirichlet problem for the Helmholtz equation, we obtain that $w_i(x)=0$ in $\mathbb{R}^d\backslash\overline{D}_i$. Thus, from the jump relations 
for the normal derivative of single-layer boundary integral operators in \cite{Mclean2000} and 
smoothness of $\GreenD{D_i}(x,y)-\Phi(x,y)= \Green^{\mathrm{s}}(x,y)$,
we have $\overline{\psi}=\frac{\partial w^{-}_i(x)}{\partial\nu}-\frac{\partial w^{+}_i(x)}{\partial\nu}=0$. (Here, the superscripts $+$ and $-$ indicate the limit values of function $w_i(x)$ approaching $M$ from the external and internal side, respectively.) Therefore, $\psi=0$. 

As $\lambda_j>0$ we can define functions
\ben
v_{i,j}(x):=\frac{1}{\lambda_j}\int_{\Omega}\GreenD{D_i}(x,y)(\Cov[g_i]\NopD{D_i}^*f_j)(y){\rm d}y,\;j\in\mathbb{N}, \,i=1,2,\;
x\in \mathbb{R}^d\backslash(\overline{D_1}\cup \overline{D_2})
\enn
and due to the identity \eqref{4.22}
we get
\ben
v_{1,j}|_{M}=\frac{1}{\lambda_j}\mathcal{C}\near(D_1,\Cov[g_1])f_j=\frac{1}{\lambda_j}\mathcal{C}\near(D_2,\Cov[g_2])f_j=v_{2,j}|_{M},\; j\in\mathbb{N}.
\enn
By analytic continuation we get
\be\label{3.5}
v_{1,j}=v_{2,j}\;\text{ in } \mathbb{R}^d\backslash(\overline{D_1}\cup \overline{D_2}\cup\overline{\Omega}),\;j\in\mathbb{N}.
\en
In what follows, we will prove that these functions are linearly independent in $\mathbb{R}^d\backslash(\overline{D_1}\cup \overline{D_2}\cup\overline{\Omega})$. Assume that $\sum_{j\in\mathbb{N}}c_jv_{i,j}=0$ in $\mathbb{R}^d\backslash(\overline{D_1}\cup \overline{D_2}\cup\overline{\Omega})$ for some $c_j\in\mathbb{R}$. Restricting $\sum_{j\in\mathcal{J}}c_jv_{i,j}$ to $M$, we obtain
\ben
0=\sum_{j\in\mathcal{J}}c_jv_{i,j}
=\mathcal{C}\near(D_i,\Cov[g_i])\sum_{j\in\mathbb{N}}\frac{c_j}{\lambda_j}f_j.
\enn
As $\mathcal{C}\near(D_i,\Cov[g_i])$ is injective, this implies 
$\sum_{j\in\mathbb{N}}\frac{c_j}{\lambda_j}f_j= 0$, and hence $c_j=0$.


Suppose that $D_1\neq D_2$. Then we may assume 
without loss of generality 
that $D^*:=(D_1\cup D_2)\backslash D_2$ is non-empty. Hence, $v_{2,j}=0$ on $\partial D^*$ for all $j\in\mathbb{N}$. As $v_{2,j}$ satisfies the Helmholtz equation in $D^*$, it follows that 
$v^j_{D_2,q_2}$ is a Dirichlet eigenfunction for the negative Laplacian in $D^*$ with corresponding eigenvalue $\kappa^2$
for each $j\in\mathbb{N}$. This leads to a contradiction to the Rellich selection theorem, i.e., the compactness 
of the embedding $H_0^1(D^*)\hookrightarrow L^2(D^*)$ 
since it implies that all eigenspaces of the negative 
Dirichlet-Laplacian are finite-dimensional. 
This shows that $D_1=D_2$. 

Thus, we get $\mathcal{C}\near(D,\Cov[g_1]) = \mathcal{C}\near(D,\Cov[g_2])$. If the source processes $g_i$ are spatially uncorrelated, it follows from the uniqueness results in Theorem \ref{source} that $q_1=q_2$, thus completing the proof. 
\end{proof}
We point out that the last part of the proof follows Schiffer's 
uniquenss proof for classical sound-soft inverse obstacle scattering problems (see, e.g., \cite[Thm.~5.1]{Colton2019}).

\section{Fr\'echet derivative of the forward operator and its adjoint}
Whereas we could show uniqueness for very general domains, for reconstructions we have to confine ourselves to a numerically more feasilble class of
domains. In particular, for iterative reconstruction methods the topology of the 
unknown scattering domain has to be known a-priori. 

\subsection{Formulation of the simplified forward operator}
Consider some smooth reference domain 
$D_{\mathrm{ref}}\subset\mathbb{R}^d$ (e.g., a ball or a torus). We will represent candidate 
domains $D=D_p$ by parameterizations
\[
p=P(\shape): \partial D_{\mathrm{ref}}\to \partial D_p,\qquad \shape\in\shapeman,
\]
where $p$ belongs to some shape manifold 
$\shapeman$ of admissible parameterizations. 
E.g.,\ if we confine ourselves to star-shaped domains with respect to the origin, we may 
choose $\partial D_{\mathrm{ref}}=S^{d-1}$, 
$\shapeman = \{\shape\in H^s(S^{d-1}):\shape>0\}$ for 
$s>(d+1)/2$, and $p(\hat{x})=(P(\shape))(\hat{x}):=\shape(\hat{x})\hat{x}$. A simple way to treat non-star-shaped domains is to allow all vector-valued functions 
on $\partial D_{\mathrm{ref}}$ which are either 
sufficiently small or satisfy a constraint 
preventing self-penetration. In $\mathbb{R}^2$, 
a shape manifold $\shapeman$ containing all smooth, closed curves in $\mathbb{R}^2$ has 
been described in \cite{Eckhardt2019(2)}. 

We assume that $\shapeman$ is an open 
subset of a Hilbert space or, more generally, a smooth Hilbert-space manifold and that 
\[
P:\shapeman \to H^{s}(\partial D_{\mathrm{ref}}; \mathbb{R}^d)
\]
with $s>(d+1)/2$ is differentiable. Note that by Sobolev embeddings, the condition on $s$ implies 
that $P(\shape)\in C^1(\partial D_{\mathrm{ref}}; \mathbb{R}^d)$ for all $\shape\in\shapeman$. 

For numerical reconstructions, we not only confine ourselves to a class of obstacles parameterized via 
$\shape\in\shapeman$, but we also assume uncorrelated sources $g$ with covariance operator 
$\Cov[g]= M_q$ for some $q\in L^{\infty}(\Omega)$. 

Finally, we have to specify a norm, ideally a Hilbert norm, in the image space of the forward operator, 
the space of covariance operators. 
A natural choice is the space $\HS(L^2(M))$ of Hilbert-Schmidt operators on $L^2(M)$ introduced in the appendix. 

In summary, the forward operator $\mathcal{C}(\shape,q):=\mathcal{C}(D_p,M_q)$ in our simplified, overloaded 
notation is given by  
\be\label{4.4}
\mathcal{C}: \shapeman\times L^{\infty}{(\Omega)} \to 
{\rm HS}(L^2(M)),\qquad 
\mathcal{C}\near(\shape,q)=
 \NopD{D_{p}} M_q \NopD{D_{p}}^*.
\en

If $X$ is an $L^2$-space, then the Hilbert-Schmidt norm of operators on this space is given by 
the $L^2$-norm of the Schwartz kernels, in particular 
\ben
\|\mathcal{C}\near(\shape,q)\|_{{HS}}=\|c_{D_p,M_q}\|_{L^2(M\times M)}.
\enn
As the covariance function $c_{D_p,M_q}$ is the input of our inverse problems, this shows 
that $\HS$ is a natural choice for the image space. 

\subsection{Fr\'echet derivative} 
By the chain rule and well known results on 
domain derivatives (see \cite{Hohage1997,Kress1999,Hettlich1995,Hettlich1998}), 
the operator 
\[
\tilde{\mathcal{G}}: \shapeman\to 
L(L^{2}(\Omega), L^2(M)),\qquad  
\tilde{\mathcal{G}}(\shape):=
\NopD{D_{p}}
\] 
is Fr\'echet differentiable with derivative 
$\tilde{\mathcal{G}}'[\shape](\partial\shape)
= \NopD{D_{p}}' \circ P'[\shape]\partial\shape$, 
where $\NopD{D_{p}}'(h): L^{2}(\Omega)\rightarrow L^2(M)$ is given by 
\be\label{eq:shape_der}
\paren{\paren{\NopD{D_{p}}'(h)}\varphi}(x)
=\int_{\Omega}u'_{h,y}(x)\varphi(y)\dy. 
\en
Here, $u_{h,y}'$ is the unique solution to the following exterior boundary value problem
\be
\label{4.2}
\begin{cases}
\Delta u_{h,y}'+\kappa^2 u_{h,y}'=0\quad & {\rm in\; }\mathbb{R}^d\backslash\overline{D}_{p},\\
u_{h,y}'=-(h\cdot\nu)\;\frac{\partial \GreenD{D_{p}}(\cdot,y)}{\partial\nu}\quad &{\rm on\;}\partial D_{p},\\
\displaystyle{\lim_{r\to\infty}}r^{\frac{d-1}{2}}\left(\frac{\partial u_{h,y}'}{\partial r}-i\kappa u_{h,y}'\right)=0\quad&{\rm for\;} r=|x|,
\end{cases}
\en
and $\GreenD{D_{p}}$ is the Dirichlet Green's function of Problem \ref{2.3} associated with obstacle $D_{p}$.





Based on \eqref{eq:shape_der} and the chain rule we obtain:

\begin{prop}
The forward operator defined in \eqref{4.4} is 
Fr\'{e}chet differentiable, and the derivative is given by
\be\label{eq:cprime_chain_rule}
\mathcal{C}\near'[\shape,q](\partial \shape,\partial q)
= 2\mathrm{Re} \paren{\paren{\NopD{D_p}'
\paren{ P'[\shape]\partial\shape}} M_{q} \NopD{D_p}^{*}}
+\mathcal{C}\near(\shape,\partial q),
\en
where $\mathrm{Re}:\mathrm{HS}(L^2(M))\to \mathrm{HS}(L^2(M)))$ is given 
by $\mathrm{Re}(\mathcal{K}):=\frac{1}{2}\paren{\mathcal{K}+\mathcal{K}^*}$.
Moreover, 
\be\label{eq:Gprime}
\NopD{D_p}'(h)
=\mathcal{G}_{\partial D_p\to M} M_{-h\cdot\nu} \mathcal{G}_{\Omega\to\partial D_p} \qquad \text{for }h\in L^2(\partial D)
\en
where  $\mathcal{G}_{\Omega\to\partial D_p}: L^2(\Omega)\to L^2(\partial D_p)$ and 
$\mathcal{G}_{\partial D_p\to M}:L^2(\Omega)\to L^2(\partial D_p)$ are the integral operators with kernels 
$\frac{\partial\GreenD{D_{p}}(x,y)}{\partial\nu(x)}$, $x\in \partial D_p$, $y\in\Omega$ amd 
$\frac{\partial\GreenD{D_{p}}(x,y)}{\partial\nu(y)}$, $y\in \partial D_p$, $x\in M$, respectively. 
\end{prop}
\begin{proof}
Eq.~\eqref{eq:cprime_chain_rule} is a straightforward application of the chain rule. 
Eq.~\eqref{eq:Gprime} follows from the characterization \eqref{eq:shape_der}, \eqref{4.2} after noting 
that $\mathcal{G}_{\partial D_p\to M}f= w|_M$ is the restriction to $M$ of the radiating solution $w$ to 
the Helmholtz equation in $\mathbb{R}^d\setminus \overline{D_p}$ with boundary values $w|_{\partial D_p} =f$. 
(The boundary condition can be interpreted in the sense of a trace operator if $f\in H^{1/2}(\partial D_p)$, 
otherewise, for $f\in L^2(\partial D_p)$, see \cite{kersten:80}.)
\end{proof}

\subsection{Adjoint of Fr\'echet derivative}\label{sec:adjoint} 
In regularization methods we also need the adjoint of the Fr\'echet derivative of the forward operator. 
To characterize these, we will use the fact shown in \cite[Prop.~3]{Mueller2023} that for any 
open set $A\subset\mathbb{R}^d$, or more generally any finite dimensional Riemannian manifold $A$, there exists 
a unique operator
\[
\Diag= \Diag_A: S_1(L^2(A))\to L^1(A)
\]
such that 
\[
(\Diag \mathcal{K})(x) = K(x,x),\qquad x\in A
\]
for all operators $\mathcal{K}\in S_1(L^2(A))$ with continuous Schwartz kernel $K$. Moreover, 
$\operatorname{tr} \mathcal{K} = \int_A (\Diag K)(x)\,\mathrm{d}x$. 
From this we can easily derive the following infinite dimensional 
analog of the fact that the mappings $\diag:\mathbb{C}^n\to\mathbb{C}^{n\times n}$ mapping a vector to a diagonal matrix and the mapping 
$\Diag:\mathbb{C}^{n\times n}\to \mathbb{C}^n$ taking the diagonal 
of a matrix are adjoint to each other with respect to the standard 
inner product in $\mathbb{C}^n$ and the Frobenius inner product 
in $\mathbb{C}^{n\times n}$. With this we will be able to easily 
derive the form of the adjoint. 

\begin{lem}\label{lem:Diag}
The adjoint of the operator $\Diag$ above is given by 
\[
\Diag^*: L^{\infty}(A)\to S_1(L^2(A))'=L(L^2(A)), \qquad \Diag^*(f)=M_f
\]
with a multiplication operator $M_f(g)= f\cdot g$ for $g\in L^2(A)$. In particular, the adjoint 
of the operator $\mathcal{D}:=\Diag^*$ restricted to $S_1(L^2(A))\subset S_1(L^2(A))''$ coincides 
with $\Diag$ and hence takes values in $L^1(A)\subset L^{\infty}(A)'$.
\end{lem}

\begin{proof}
For the duality $S_1(X)'=L(X)$ in a Hilbert space $X$ with duality pairing $\langle \mathcal{K}_1,\mathcal{K}_2\rangle = \operatorname{tr}(\mathcal{K}_1^*\mathcal{K}_2)$ 
for $\mathcal{K}_1\in L(X)$ and $\mathcal{K}_2\in S_1(X)$ we refer to \cite{Reed1980}. For $\mathcal{K}\in S_1(L^2(A))$ with kernel $K$ and $f\in L^{\infty}(A)$
we have 
\begin{align*}
\langle f, \Diag \mathcal{K} \rangle 
&= \int_A \overline{f(x)}  (\Diag \mathcal{K})(x) \,\mathrm{d}x
= \int_A \overline{f(x)}K(x,x)\,\mathrm{d}x\\
&= \int_A \Diag(M_f^*\mathcal{K})\,\mathrm{d}x
= \operatorname{tr} (M_f^*\mathcal{K}) = \langle M_f,\mathcal{K}\rangle.
\end{align*}
This shows that $\Diag^* f = M_f$. 
\end{proof}

\begin{prop}
The adjoint operator $\mathcal{C}\near'[\shape,q]^*: {\rm HS}(L^2(M))\rightarrow T_{\shape}\shapeman\times L^{\infty}(\Omega)'$ of $\mathcal{C}\near'[\shape,q]$ takes values in $T_{\shape}\shapeman\times L^1(\Omega)$ and is given by
\be\label{eq:adjoint}
\mathcal{C}\near'[\shape,q]^{*}\mathcal{K}=\left(\begin{array}{l}
2 P'[\shape]^*(\NopD{D_{p}}')^{*}({\rm Re}(\mathcal{K})\NopD{D_{p}}M_q) \\
\Diag_{\Omega}(\NopD{D_{p}}^*\mathcal{K}\NopD{D_{p}})
\end{array}\right),
\qquad \mathcal{K}\in\mathrm{HS}(L^2(M))
\en
where
\ben
(\NopD{D_{p}}')^{*}(\mathcal{J}) = -\nu \Diag_{\partial D_p}(\mathcal{G}_{\partial D_p\to M}^*
\mathcal{J} \mathcal{G}_{\Omega\to\partial D_p}^*),
\qquad \mathcal{J}\in \mathrm{HS}(L^2(\Omega),L^2(M)).
\enn
\end{prop}

\begin{proof}
We prove \eqref{eq:adjoint} only for the special case that $\rho$ is fixed 
as the general case then follows by straightforward computations. 
Writing  $q\mapsto\mathcal{C}\near(\shape,q) = \NopD{D_p}\Diag^*(q)\NopD{D_p}^*$, we find from Lemma \ref{lem:Diag} 
and the identity 
\[
\langle Y,AXB\rangle  =
\operatorname{tr}(Y^*AXB) = \operatorname{tr}(BY^*AX)
= \operatorname{tr}((B^*YA^*)^*X)) 
= \langle B^*YA^*,X\rangle
\]
using the invariance of the trace under cyclic permutations that 
\[
\mathcal{C}\near(\shape,\cdot)^*(\mathcal{K})
= \Diag_{\Omega}^{**}\left(\NopD{D_p}^*\mathcal{K}\NopD{D_p}\right)
= \Diag_{\Omega}\left(\NopD{D_p}^*\mathcal{K}\NopD{D_p}\right)
\]
for $\mathcal{K}\in\mathrm{HS}(L^2(M))$. 
In the last line we have used that $\NopD{D_p}$ is Hilbert-Schmidt 
as its kernel is square integrable, and hence 
$\NopD{D_p}^*\mathcal{K}\NopD{D_p}\in S_1(L^2(\Omega)$
as products of Hilbert-Schmidt operators are trace class. 

The derivation of the formula for $(\NopD{D_{p}}')^{*}$ is analogous. 
\end{proof}

\section{Inversion methods}

In this section, we discuss iterative regularization methods for reconstructing the source strength, shape and location of the obstacle, as well as for their simultaneous recovery.

\subsection{Discrete forward operator}

We discretize the source region $\Omega$ into $\Nsrc$ points $y_n,\;(n=1,\dots,\Nsrc)$ where $y_n\in\Omega_n$ and $\Omega_n$ are disjoint subdomains such that $\Omega=\bigcup_{n}\Omega_{n}$. The random source process $g$ can then be approximated by its projection
\ben
(P_{\Nsrc} g)(y)=\sum_{n=1}^{\Nsrc}\langle g,\phi_n\rangle\phi_n(y)
= \sum_{n=1}^{\Nsrc}\underline{\pi}_n\phi_n(y),\qquad y\in \Omega,
\enn
where $\phi_n=1_{\Omega_n}|\Omega_n|^{-1/2}$ are the $L^2$-normalized  
indicator functions of $\Omega_n$, and the coefficients 
$\underline{\pi}_n:=\langle g,\phi_n\rangle$ satisfy 
$\EV[\underline{\pi}_n]=0$ and $\EV(|\underline{\pi}_n|^2)=|\Omega_n|^{-1}\int_{\Omega_n}q(y)\dy\approx q(y_n)$. 
From Proposition \ref{solution}, the total random field $u$ at $M$ can be approximated by 
\ben
u(x)\approx (\NopD{D_p} P_{\Nsrc} g)(x)
=\sum_{n=1}^{\Nsrc}\frac{\underline{\pi}_{n}}{|\Omega_n|^{1/2}}
\int_{\Omega_n}G_{D_p}(x,y)\dy,\qquad x\in M.
\enn
Let $x_l\in M$, $l=1, \ldots, {\Nmeas}$ denote the measurement points, which 
are chosen (approximately) equidistant on the measurement curve or 
surface $M$. To obtain approximations 
$\underline{u}_l\approx u(x_l)$, we further approximate 
$\int_{\Omega_n}G_{D_p}(x_l,y)\dy\approx |\Omega_n|G_{D_p}(x_l,y_n)$ 
to obtain
\ben
\underline{u}=\underline{\NopD{D_{p}}}\underline{\pi}
\qquad \text{with}\qquad
(\underline{\NopD{D_{p}}})_{ln}:=|\Omega_n|^{1/2}G_{D_p}(x_l,y_n), 
\enn
$\underline{u}:=(\underline{u}_1,\cdots,\underline{u}_{\Nmeas})^\top$, 
$\underline{\pi}:=(\underline{\pi}_1,\cdots,\underline{\pi}_{\Nsrc})^{\top}$, 
and $\underline{\NopD{D_{p}}}\in\mathbb{C}^{{\Nmeas}\times {\Nsrc}}$. 
The Green function $G_{D_p}$ is evaluated using boundary integral 
equation methods as described in \cite{Colton2019}. 
Based on Lemma \ref{lem:cov_fct_op}, we obtain the following discrete forward operator:
\ben
\underline{\mathcal{C}\near(\shape,q)}
:= \mathrm{Cov}(\underline{u})=\mathrm{Cov}\left(\underline{\NopD{D_{P(\shape)}}}\underline{\pi}\right)
=\underline{\NopD{D_{P(\shape)}}}\underline{M_q}\underline{\NopD{D_{P(\shape)}}^{*}},
\enn
where  $\underline{M_q}$ is the diagonal matrix defined by
\ben
\underline{M_q}=\mathrm{diag}(q(y_1),\cdots,q(y_{\Nsrc})).
\enn

\subsection{Synthetic data}
To generate synthetic data, we draw $\Nsample$ independent 
(complex-valued) proper normally distributed samples 
\[
\underline{\pi}^{(j)}\sim \mathcal{CN}(0,\underline{M}_q),
\qquad j=1,\dots,\Nsample
\]
of the discretized source process and compute the corresponding 
primary data simulating measurements of total fields by 
\ben
\underline{u}^{(j)}:=\underline{\NopD{D_{p}}} \underline{\pi}^{(j)}
+ \underline{\epsilon}^{(j)}. 
\enn
Here $\underline{\epsilon}^{(j)}\sim \mathcal{CN}(0,\beta I)$ model independent measurement errors, and also independent of 
$\underline{\pi}^{(j)}$. The correlation matrix 
$\underline{\mathcal{C}}^{\mathrm{obs}}\in \mathbb{C}^{{\Nmeas}\times {\Nmeas}}$ 
with entries 
$\underline{\mathcal{C}}^{\mathrm{obs}}_{lm} \approx 
c_{D_p,M_q}(x_l,x_m)$, which simulates the data in Problem 
\ref{problem} and serves 
as input data of our inversion methods,  is then computed by 
\ben
\underline{\mathcal{C}}^{\mathrm{obs}}
:=\frac{1}{\Nsample}\sum\nolimits_{j=1}^{\Nsample}\underline{u}^{(j)}(\overline{\underline{u}^{(j)}})
\enn
where the $\underline{u}^{(j)}$ are viewed as column vectors and 
the $(\overline{\underline{u}^{(j)})}$ as row vectors. The errors 
$\underline{\epsilon}^{(j)}$ are chosen small such that the finiteness 
of $\Nsample$ is the main source of errors. As an additional precaution against inverse crimes, different discretization parameters are chosen 
in the integral equation method for generating synthetic data and for 
reconstructions. 

\subsection{Data fidelity term}\label{weightingmatrix}
The most natural choice of the (symmetric positive definite) weight matrix 
$W\in \mathbb{C}^{\Nmeas^2\times\Nmeas^2}$ in a quadratic data fidelity term
\[
\|W^{-1/2}(\underline{C}(\shape,q)-\underline{C}^{\mathrm{obs}})\|_{\HS}^2
\]
is $W= \Cov(\underline{C}^{\mathrm{obs}})$ as quadratic approximation 
of the negative log-likelihood functional. Note that 
$W$ involves forth order moments of the primary data 
$\underline{u}^{(j)}\in \mathbb{C}^{\Nmeas}$ and that it is a dense matrix 
with ${\Nmeas}^4$ entries, which may seem computationally infeasible for large 
${\Nmeas}$ at first sight. However, since 
%
%
\ben
{\rm Cov}(\underline{\mathcal{C}}^{\mathrm{obs}}_{ml},
\underline{\mathcal{C}}^{\mathrm{obs}}_{m'l'})&=&{\rm Cov}\left(\frac{1}{{\rm \Nsample}}\sum_{j=1}^{{\rm \Nsample}}\underline{u}_m^{(j)}\overline{\underline{u}^{(j)}_l},\frac{1}{{\rm \Nsample}}\sum_{j=1}^{{\rm \Nsample}}\underline{u}_{m'}^{(j)}\overline{\underline{u}^{(j)}_{l'}}\right)\\\
&=&\EV\left(\underline{u}_m^{(1)}\overline{\underline{u}_l^{(1)}}\overline{\underline{u}^{(1)}_{m'}}\underline{u}_{l'}^{(1)}\right)
-\EV\left(\underline{u}_m^{(1)}\overline{\underline{u}^{(1)}_l}\right)\overline{\EV\left(\underline{u}_{m'}^{(1)}\overline{\underline{u}_{l'}^{(1)}}\right)},
\enn
and according to Isserlis' theorem for products of normally distributed random variables (\cite{Isserlis1916}) we have
\ben 
\EV\left(\underline{u}_m^{(1)}\overline{\underline{u}^{(1)}_l }\overline{\underline{u}_{m'}^{(1)}}\underline{u}_{l'}^{(1)}\right)
&=&\EV\left(\underline{u}_m^{(1)}\overline{\underline{u}^{(1)}_l}\right)\EV\left(\overline{u^{(1)}_{m'}} \underline{u}_{l'}^{(1)}\right)+\EV\left(\underline{u}_m^{(1)}\overline{\underline{u}_{m'}^{(1)}}\right)\EV\left(\overline{\underline{u}^{(1)}_l} \underline{u}_{l'}^{(1)}\right)\\
&&+\EV\left(\underline{u}_m^{(1)} \underline{u}_{l'}^{(1)}\right)\EV\left(\overline{\underline{u}^{(1)}_l}\overline{\underline{u}^{(1)}_{m'}}\right).
\enn
By the properties of $\underline{\epsilon}^{(1)}$, we have 
$\EV(\underline{u}_m^{(1)}\overline{\underline{u}_{m'}^{(1)}})
= \EV((\underline{\NopD{D_{p}}} \underline{\pi}^{(1)})_m(\overline{\underline{\NopD{D_{p}}} \underline{\pi}^{(1)}})_{m'} + \beta\delta_{mm'}$. Moreover, 
as $\EV(XY)=0$ for proper complex random variables $X$ and $Y$, 
we have $\EV(\underline{u}_m^{(1)}\underline{u}_{l'}^{(1)})=0$
such that we end up with 
\begin{align}\label{eq:weight_datafid}
W=(\underline{\mathcal{C}}^{\mathrm{obs}}+\beta\mathbf{I})\otimes(\underline{\mathcal{C}}^{\mathrm{obs}}+\beta\mathbf{I}), 
\end{align}
where $\otimes$ denotes the Kronecker product of matrices.  
This implies $W^{\alpha}=(\underline{\mathcal{C}}^{\mathrm{obs}}+\beta\mathbf{I})^{\alpha}\otimes(\underline{\mathcal{C}}^{\mathrm{obs}}+\beta\mathbf{I})^{\alpha}$ for $\alpha\in \{-1/2,-1\}$ and 
$W^{-1}(A) = (\underline{\mathcal{C}}^{\mathrm{obs}}+\beta\mathbf{I})^{-1}A
(\overline{\underline{\mathcal{C}}^{\mathrm{obs}}+\beta\mathbf{I}})^{-1}$,  
which yields 
the use of $W$ computationally feasible even for larger dimensions ${\Nmeas}$.



\subsection{Reconstruction methods}
In this section we review the reconstruction methods used in this 
paper. 
\paragraph*{Source strength reconstruction}


For a stable solution of the ill-posed linear equation 
$\mathcal{C}\near(q)-\mathcal{C}^{\rm {obs}}$ we use 
Tikhonov regularization 
incorporating the weight $W$ from \eqref{eq:weight_datafid} 
in the data fidelity term and an $H^1$ penalty term:
\be
\label{6.5}
\widehat{q}_{\alpha}=\displaystyle{{\mathrm{argmin}}_{q\in H^1(\Omega)}}\left[\|W^{-1/2}(\mathcal{C}\near(q)-\mathcal{C}^{\rm {obs}})\|_{\mathrm{HS}}^2+\alpha\|q\|^2_{H^1}\right],
\en
where $\alpha>0$ is a regularization parameter. 
This minimization problem is solved by the conjugate gradient (CG) iteration.

\paragraph*{BERGN for obstacle reconstruction}
To recover non star-shaped obstacles in $\mathbb{R}^2$, we need to work within the space of closed curves on a Riemannian manifold, as discussed in \cite{Mumford2006}. For this purpose, we use the iterative Bending Energy Regularized Gauss-Newton Method (BERGN).
In each Newton-type iteration, we compute an update direction $\partial\shape\in\mathcal{T}_{\shape_n}\mathcal{M}$ in the tangent space $\mathcal{T}_{\shape_n}\mathcal{M}$ 
of shape manifold $\mathcal{M}$ at the point $\shape_n\in\mathcal{M}$ by solving the quadratic minimization problem 
\be
\label{5.7}
&&\|W^{-1/2}(\mathcal{C}\near'(\shape_n)\partial\shape+\mathcal{C}\near(\shape_n)-\mathcal{C}^{\mathrm{obs}})\|_{\mathrm{HS}}^2\\\no
&&+\alpha_n\langle\mathrm{ Hess}\;\varepsilon(\shape_n)\partial\shape,\partial\shape\rangle_{\shape_n}+\alpha_n\langle {\rm grad}\;\varepsilon(\shape_n)\partial\shape,\partial\shape\rangle_{\shape_n}=\mathrm{min}!
\en
Here $\varepsilon(\shape_n)$ describes the bending energy, and 
$\langle\cdot,\cdot\rangle_{\shape_n}$ denotes the Riemannian metric. The we  compute the next iterate 
\be 
\label{update}
\shape_{n+1}:\approx\exp_{\shape_n}(\partial\shape)
\en
using an approximation of the Riemannian exponential map. 
Here $\alpha_n =\alpha_0 c_{\alpha}^n$ for some 
$c_{\alpha}\in (0,1)$, and the minimization 
problems in each Newton step are solved again by the conjugate gradient method. 





\paragraph*{BERGN for source strength and obstacle reconstruction}

To simultaneously recover the source strength and the shape of a non star-shaped obstacle, we solve the quadratic optimization problem 
\be\no
&&\|W^{-1/2}(\mathcal{C}\near'(\shape_n,q_n)(\partial\shape,\partial q)+\mathcal{C}\near(\shape_n,q_n)-\mathcal{C}^{\mathrm{obs}})\|_{\mathrm{HS}}^2
+\alpha_n\langle \mathrm{Hess}\;\varepsilon(\shape_n)\partial\shape,\partial\shape\rangle_{\shape_n}\\\label{5.9}
&&+\alpha_n\langle \mathrm{grad}\;\varepsilon(\shape_n)\partial\shape,\partial\shape\rangle_{\shape_n}+\alpha_n\|\partial q\|_{H^1}^2=\mathrm{min}!
\en
in each Newton step and use formula \eqref{update} to get $\shape_{n+1}$ and $q_{n+1} := q_n + \partial q$ to update the source 
strength.


\paragraph*{Newton-CG method} 
Following \cite{hanke:97b} we solve the linearized least squares problem 
\[
\|W^{-1/2}(\mathcal{C}\near'(\shape_n)(\partial\shape)+\mathcal{C}\near(\shape_n)-\mathcal{C}^{\mathrm{obs}})\|_{\mathrm{HS}}^2 = \min! 
\]
for the update $\partial\shape = \shape_{n+1}-\shape_n$ 
by the CG method and use the regularizing properties of CG with 
early stopping by terminating the inner iteration 
at the first iterate $\partial\shape_k$ for which 
\[
\|W^{-1/2}(\mathcal{C}\near'(\shape_n)\partial\shape_k+\mathcal{C}\near(\shape_n)-\mathcal{C}^{\mathrm{obs}})\|_{\mathrm{HS}} \leq 0.8 \|W^{-1/2}(\mathcal{C}\near(\shape_n)-\mathcal{C}^{\mathrm{obs}})\|.
\]

\subsection{Reconstruction results}

In this section, we present several numerical examples to demonstrate the effectiveness of the proposed methods for both two-dimensional and three-dimensional cases of Problem \eqref{ui}-\eqref{us}, employing Dirichlet boundary conditions. For the numerical simulations in $\mathbb{R}^d$, we utilize near-field correlation data, selecting the measurement domain $M$ to be either a circle or a sphere with varying radius $R$ centered at $(0,0)$. Meanwhile, we set $\beta=0.01$. These examples include the reconstruction of both star-shaped and non-star-shaped obstacles, along with the inversion of source strength.


In the following examples the stopping index for the regularized 
Newton methods and the regularization parameter in Tikhonov regularization were chosen by trial and error. The development 
of theory-based, fully data driven stopping and parameter choice rules for inverse problem with 
correlation data with $\Nsample$ and $\beta$ parameterizing 
the noise level(s) is beyond the scope of this paper and will be addressed in future research. 
If finiteness of $\Nsample$ is the dominating source of 
error, the noise level is proportional to $\Nsample^{-1/2}$, and $\alpha \sim \Nsample^{-1/2}$ is 
a plausible a-priori parameter choice rule in 
Tikhonov regularization. Similarly, $\alpha_{N_{\mathrm{stop}}} \sim \Nsample^{-1/2}$ would 
be a canonical asymptotic a-priori choice of the stopping 
index $N_{\mathrm{stop}}$ in regularized Gau{\ss}-Newton methods, but both involve free constants hidden in the 
symbol $\sim$. 




\begin{example} This example evaluates Tikhonov regularization with CG minimization for reconstructing the source strength function using correlation data. 
Figure \ref{f1} shows the method's effectiveness in solving the passive inverse source problem.
\end{example}

\begin{figure}[htbp]
\begin{minipage}[t]{0.3\linewidth}
\centering
%
%
\begin{tikzpicture}

\begin{axis}[%
width=1.5in,
height=1.5in,
at={(0.867in,0.478in)},
scale only axis,
point meta min=1.84369026740875,
point meta max=36.8592528326288,
axis on top,
separate axis lines,
every outer x axis line/.append style={black},
every x tick label/.append style={font=\color{black}},
every x tick/.append style={black},
xmin=-3,
xmax=3,
xtick={-3,-1.5,0,1.5,3},
tick align=outside,
every outer y axis line/.append style={black},
every y tick label/.append style={font=\color{black}},
every y tick/.append style={black},
ymin=-3,
ymax=3,
ytick={-3,-1.5,0,1.5,3},
axis background/.style={fill=white},
title style={font=\normalfont},
title={True source}
]
\addplot [forget plot] graphics [xmin=-2.55, xmax=-1.95, ymin=-1.56521739130435, ymax=1.56521739130435] {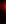};
\addplot [forget plot] graphics [xmin=1.95, xmax=2.55, ymin=-1.56521739130435, ymax=1.56521739130435] {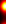};
\addplot [color=orange, line width=1.5pt, forget plot]
  table[row sep=crcr]{%
1.08571428571429	0\\
1.09136835510355	0.107490400873748\\
1.08554740833152	0.215928804894075\\
1.06785865379288	0.323931381189473\\
1.03809777026735	0.429994175514005\\
0.996277926166649	0.532521646038035\\
0.942653886426241	0.629861189861619\\
0.877740113431713	0.72034269495295\\
0.80232197656312	0.80232197656312\\
0.717459420328546	0.874226833136639\\
0.624482689979699	0.934604392508463\\
0.524979958622756	0.982168421409814\\
0.420776919094217	1.01584534481082\\
0.313908568644468	1.034817868821\\
0.206583486598323	1.0385653205968\\
0.101140832328559	1.02690010377643\\
6.12323399573677e-17	1\\
-0.094397744706266	0.958436386208721\\
-0.179657392099061	0.903198702351946\\
-0.25351153602328	0.835715535138107\\
-0.313920934389218	0.757872177315284\\
-0.359203737915778	0.672022926663303\\
-0.388206950208214	0.580992758788408\\
-0.400537108174766	0.488055878439765\\
-0.396866176580776	0.396866176580776\\
-0.379314824191471	0.311295631260842\\
-0.35187319190977	0.235114150090586\\
-0.320728646175497	0.171433032999225\\
-0.294223939476191	0.121871546105879\\
-0.282025273840945	0.0855514315131033\\
-0.293128695042437	0.0583069226778011\\
-0.3328566511829	0.03278351869177\\
-0.4	4.89858719658941e-17\\
-0.485965593990855	-0.0478634333354572\\
-0.577027792407777	-0.114777964231689\\
-0.658686595878966	-0.199810394396494\\
-0.719689266619565	-0.29810505492817\\
-0.753912002510616	-0.402974360868885\\
-0.760046552934892	-0.507846870495294\\
-0.740134059603637	-0.60741232508662\\
-0.697974150476163	-0.697974150476163\\
-0.637926507864967	-0.777315698899487\\
-0.564191974252096	-0.844372959915306\\
-0.480466195378081	-0.89888902786548\\
-0.389826354403461	-0.941124071771295\\
-0.294742856625474	-0.971636983534636\\
-0.197148170033835	-0.991130781011441\\
-0.0985256369448776	-1.00034757944903\\
-1.83697019872103e-16	-1\\
0.0975782861743081	-0.990728965658613\\
0.193557807705512	-0.973080810687221\\
0.287420642751785	-0.947498879281725\\
0.378726350905203	-0.914326292783411\\
0.467065175909602	-0.873817483856129\\
0.552020198856643	-0.826156610716602\\
0.633138276722282	-0.771481222370796\\
0.709909654841646	-0.709909654841646\\
0.781756127146142	-0.641570673149193\\
0.848027536160656	-0.566633884029886\\
0.908006268661578	-0.48533946211274\\
0.960919229442548	-0.398025777180208\\
1.00595658243905	-0.305153593135863\\
1.04229635222288	-0.207325634931836\\
1.06913379460469	-0.105300487807186\\
1.08571428571429	0\\
};
\end{axis}
\end{tikzpicture}%
\end{minipage}\qquad\qquad
\begin{minipage}[t]{0.3\linewidth}
\centering
%
%
\begin{tikzpicture}

\begin{axis}[%
width=1.5in,
height=1.5in,
at={(0.867in,0.478in)},
scale only axis,
point meta min=0.42638151211505,
point meta max=33.5427318802976,
axis on top,
separate axis lines,
every outer x axis line/.append style={black},
every x tick label/.append style={font=\color{black}},
every x tick/.append style={black},
xmin=-3,
xmax=3,
xtick={-3,-1.5,0,1.5,3},
tick align=outside,
every outer y axis line/.append style={black},
every y tick label/.append style={font=\color{black}},
every y tick/.append style={black},
ymin=-3,
ymax=3,
ytick={-3,-1.5,0,1.5,3},
axis background/.style={fill=white},
title style={font=\normalfont},
title={Reconstructed source},
colormap/hot2,
colorbar,
colorbar style={separate axis lines, every outer x axis line/.append style={black}, every x tick label/.append style={font=\color{black}}, every x tick/.append style={black}, every outer y axis line/.append style={black}, every y tick label/.append style={font=\color{black}}, every y tick/.append style={black}}
]
\addplot [forget plot] graphics [xmin=-2.55, xmax=-1.95, ymin=-1.56521739130435, ymax=1.56521739130435] {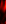};
\addplot [forget plot] graphics [xmin=1.95, xmax=2.55, ymin=-1.56521739130435, ymax=1.56521739130435] {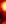};
\addplot [color=orange, line width=1.5pt, forget plot]
  table[row sep=crcr]{%
1.08571428571429	0\\
1.09136835510355	0.107490400873748\\
1.08554740833152	0.215928804894075\\
1.06785865379288	0.323931381189473\\
1.03809777026735	0.429994175514005\\
0.996277926166649	0.532521646038035\\
0.942653886426241	0.629861189861619\\
0.877740113431713	0.72034269495295\\
0.80232197656312	0.80232197656312\\
0.717459420328546	0.874226833136639\\
0.624482689979699	0.934604392508463\\
0.524979958622756	0.982168421409814\\
0.420776919094217	1.01584534481082\\
0.313908568644468	1.034817868821\\
0.206583486598323	1.0385653205968\\
0.101140832328559	1.02690010377643\\
6.12323399573677e-17	1\\
-0.094397744706266	0.958436386208721\\
-0.179657392099061	0.903198702351946\\
-0.25351153602328	0.835715535138107\\
-0.313920934389218	0.757872177315284\\
-0.359203737915778	0.672022926663303\\
-0.388206950208214	0.580992758788408\\
-0.400537108174766	0.488055878439765\\
-0.396866176580776	0.396866176580776\\
-0.379314824191471	0.311295631260842\\
-0.35187319190977	0.235114150090586\\
-0.320728646175497	0.171433032999225\\
-0.294223939476191	0.121871546105879\\
-0.282025273840945	0.0855514315131033\\
-0.293128695042437	0.0583069226778011\\
-0.3328566511829	0.03278351869177\\
-0.4	4.89858719658941e-17\\
-0.485965593990855	-0.0478634333354572\\
-0.577027792407777	-0.114777964231689\\
-0.658686595878966	-0.199810394396494\\
-0.719689266619565	-0.29810505492817\\
-0.753912002510616	-0.402974360868885\\
-0.760046552934892	-0.507846870495294\\
-0.740134059603637	-0.60741232508662\\
-0.697974150476163	-0.697974150476163\\
-0.637926507864967	-0.777315698899487\\
-0.564191974252096	-0.844372959915306\\
-0.480466195378081	-0.89888902786548\\
-0.389826354403461	-0.941124071771295\\
-0.294742856625474	-0.971636983534636\\
-0.197148170033835	-0.991130781011441\\
-0.0985256369448776	-1.00034757944903\\
-1.83697019872103e-16	-1\\
0.0975782861743081	-0.990728965658613\\
0.193557807705512	-0.973080810687221\\
0.287420642751785	-0.947498879281725\\
0.378726350905203	-0.914326292783411\\
0.467065175909602	-0.873817483856129\\
0.552020198856643	-0.826156610716602\\
0.633138276722282	-0.771481222370796\\
0.709909654841646	-0.709909654841646\\
0.781756127146142	-0.641570673149193\\
0.848027536160656	-0.566633884029886\\
0.908006268661578	-0.48533946211274\\
0.960919229442548	-0.398025777180208\\
1.00595658243905	-0.305153593135863\\
1.04229635222288	-0.207325634931836\\
1.06913379460469	-0.105300487807186\\
1.08571428571429	0\\
};
\end{axis}
\end{tikzpicture}%
\end{minipage}
\qquad\qquad
\caption{Reconstructions of source strength with $\kappa=\pi$, $\Nsample=10000$, ${\Nsrc}=288$, ${\Nmeas}=32$, and $R=5$ from $\mathcal{C}^{\mathrm{obs}}$.}\label{f1}
\end{figure}

\begin{example}
This example uses correlation data to recover a non-star-shaped obstacle in $\mathbb{R}^2$, as shown in Figure \ref{f2}, which highlights the effectiveness of the BERGN method.
\end{example}

\begin{figure}[htbp]
\begin{minipage}[t]{0.3\linewidth}
\centering
\input{Dir-obstacle/obstacle}
\end{minipage}\qquad\qquad
\begin{minipage}[t]{0.3\linewidth}
\centering
\input{Dir-obstacle/obstacle3}
\end{minipage}\qquad\qquad
\caption{Reconstructions of non-star-shaped obstacle with $\kappa=2.5\pi/2$, $\Nsample=10000$, ${\Nsrc}=288$, ${\Nmeas}=32$, and $R=5$ from $\mathcal{C}^{\mathrm{obs}}$.}\label{f2} 
\end{figure}

\begin{example}
In this example, correlation data is used to simultaneously recover the source strength and a non-star-shaped obstacle in $\mathbb{R}^2$. Figure \ref{f3} demonstrates the BERGN method's accuracy and effectiveness.
\end{example}

\begin{figure}[htbp]
\begin{minipage}[t]{0.3\linewidth}
\centering
%
%
\begin{tikzpicture}

\begin{axis}[%
width=1.5in,
height=1.5in,
at={(0.867in,0.478in)},
scale only axis,
point meta min=0.678255745267558,
point meta max=13.5597613340644,
axis on top,
separate axis lines,
every outer x axis line/.append style={black},
every x tick label/.append style={font=\color{black}},
every x tick/.append style={black},
xmin=-3,
xmax=3,
xtick={-3,-1.5,0,1.5,3},
tick align=outside,
every outer y axis line/.append style={black},
every y tick label/.append style={font=\color{black}},
every y tick/.append style={black},
ymin=-3,
ymax=4,
ytick={-3,-1.5,0,1.5,3},
axis background/.style={fill=white},
title style={font=\normalfont},
title={True source and obstacle}
]
\addplot [forget plot] graphics [xmin=-2.05, xmax=-1.45, ymin=-1.56521739130435, ymax=1.56521739130435] {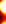};
\addplot [forget plot] graphics [xmin=1.45, xmax=2.05, ymin=-1.56521739130435, ymax=1.56521739130435] {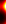};
\addplot [color=orange, line width=1.5pt, forget plot]
  table[row sep=crcr]{%
-0.149941639973572	0.501048116788004\\
-0.133843792338541	0.385289063835608\\
-0.0782051586544083	0.285974456291364\\
0.00722405612213784	0.194905919934387\\
0.10925366094537	0.110155104148614\\
0.218262542944971	0.0253720236725542\\
0.321080681214414	-0.0622345919890673\\
0.408338203540294	-0.158482280925378\\
0.466486393236096	-0.266325371693579\\
0.505157952343543	-0.385070062146835\\
0.538901652772338	-0.509333132855938\\
0.564298688438677	-0.638037232380894\\
0.578823283965763	-0.766742534451507\\
0.579268435191686	-0.894061628627929\\
0.562682703126582	-1.01569820861145\\
0.526309907313376	-1.13001183469485\\
0.465253860278564	-1.23494145972515\\
0.362134329376925	-1.34942174423232\\
0.225447661988721	-1.4617038746279\\
0.0694442188187118	-1.56137837656411\\
-0.0966477333142413	-1.63925592924197\\
-0.260348114558843	-1.68796404184409\\
-0.412263147945908	-1.69707752697889\\
-0.53968102831698	-1.66037054665443\\
-0.632718016899188	-1.56753549331994\\
-0.652966052607406	-1.4613858193447\\
-0.591792094332252	-1.37322447216791\\
-0.475362387336835	-1.2932747254772\\
-0.325433902911742	-1.21796047021955\\
-0.166427950832406	-1.13985271764912\\
-0.0212772984721446	-1.0538685255848\\
0.0865926193416984	-0.953641729425033\\
0.133275065456301	-0.83437878345467\\
0.117177063721951	-0.718619730502275\\
0.0615384745660109	-0.619305122958031\\
-0.023890767241255	-0.528236586601054\\
-0.125920352259196	-0.44348577081528\\
-0.234929250225721	-0.358702690339221\\
-0.337747374773873	-0.271096074677599\\
-0.42500490952233	-0.174848385741289\\
-0.483153087389525	-0.0670052949730876\\
-0.521824658397032	0.0517393954801681\\
-0.555568346176246	0.176002466189271\\
-0.580965393134846	0.304706565714227\\
-0.59548969932216	0.43341186778484\\
-0.595934789755774	0.560730961961262\\
-0.579349076083284	0.682367541944783\\
-0.542976269726607	0.796681168028183\\
-0.48192023028051	0.901610793058486\\
-0.378800693156139	1.01609107756565\\
-0.242114031338956	1.12837320796124\\
-0.0861105828573844	1.22804770989744\\
0.0799813639709613	1.3059252625753\\
0.243681750702025	1.35463337517742\\
0.3955967782568	1.36374686031223\\
0.523014664970328	1.32703987998776\\
0.616051646513042	1.23420482665328\\
0.636299690193864	1.12805515267803\\
0.575125722686913	1.03989380550124\\
0.458696026672357	0.959944058810537\\
0.308767528710407	0.884629803552878\\
0.149761594220362	0.806522050982455\\
0.00461091690170424	0.720537858918135\\
-0.103258958367591	0.620311062758366\\
-0.149941639973572	0.501048116788004\\
};
\end{axis}
\end{tikzpicture}%
\end{minipage}\qquad\qquad
\begin{minipage}[t]{0.3\linewidth}
\centering
\input{Dir-source-obstacle/source_obstacle5}
\end{minipage}\qquad\qquad
\caption{Reconstructions of non star-shaped obstacle and source strength with $\kappa=2.5\pi/2$, $\Nsample=10000$, ${\Nsrc}=128$, ${\Nmeas}=32$, and $R=4$ from $\mathcal{C}^{\mathrm{obs}}$.}\label{f3} 
\end{figure}

\begin{example}
In this final 3D example, the shape and location of a star-shaped obstacle in $\mathbb{R}^3$ are reconstructed using correlation data, with the source supported in the regions $[-1.5,-1]\times[-0.5, 0.5]\times[-0.5,0.5]$ and $[1, 1.5]\times[-0.5,0.5]\times[-0.5,0.5]$. Figure \ref{f4} demonstrates the Newton-CG method's effectiveness in solving this inverse problem.
\end{example}

\begin{figure}[htbp]
\begin{minipage}[t]{0.3\linewidth}
\centering
\includegraphics[width=2.5in]{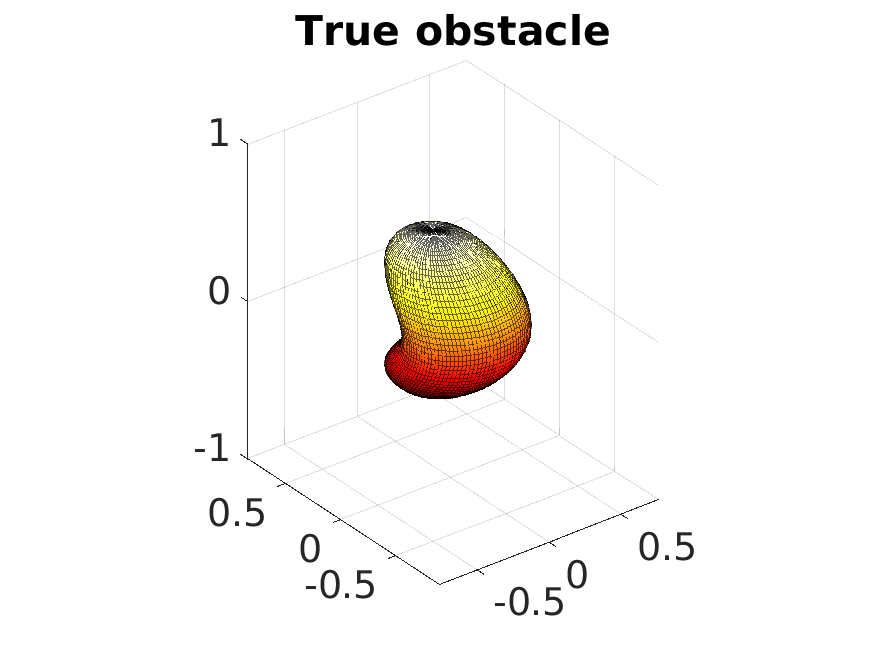}
\end{minipage}\qquad\qquad
\begin{minipage}[t]{0.3\linewidth}
\centering
\includegraphics[width=3.8in]{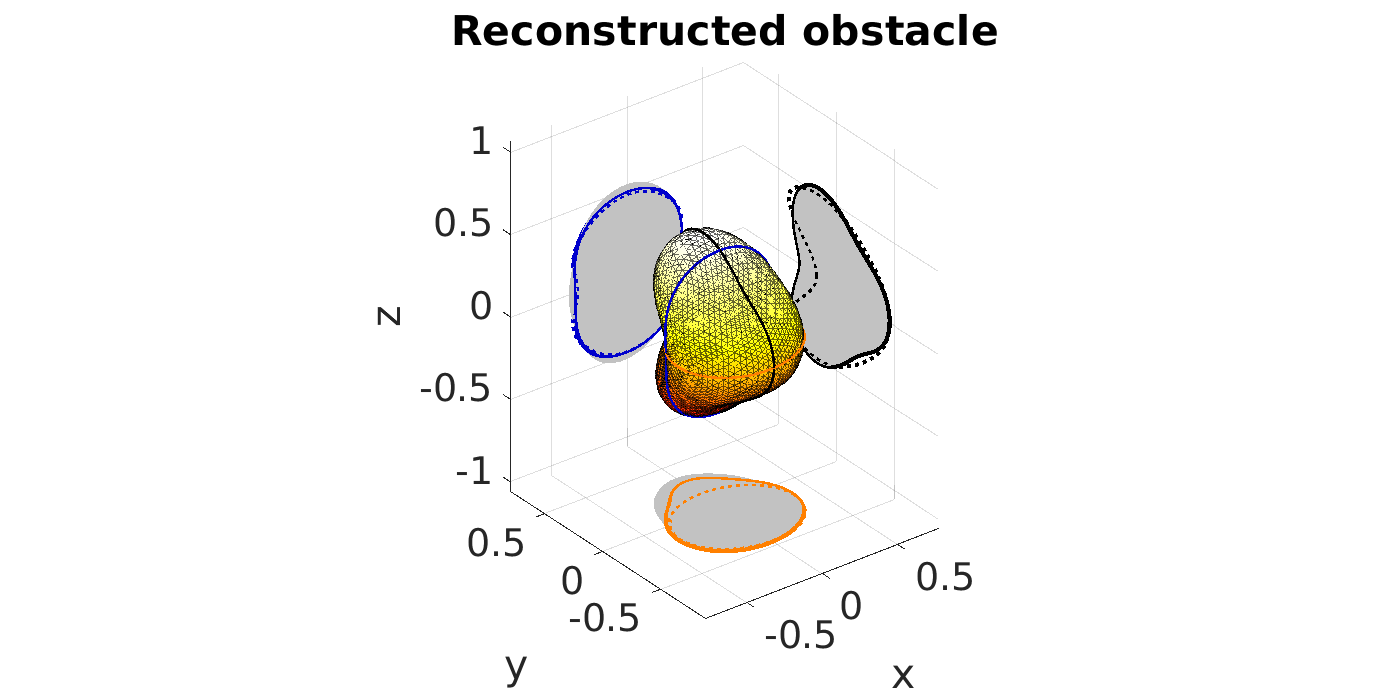}
\end{minipage}
\qquad\qquad
\caption{Reconstruction of a star-shaped obstacle for $\kappa=\pi$, $\Nsample=10000$, 
${\Nsrc}=64$ and ${\Nmeas}=42$, and $R=4$ from $\mathcal{C}^{\mathrm{obs}}$.
}\label{f4}
\end{figure}

\section{Conclusions}

In this paper, we considered passive obstacle scattering problems for the Helmholtz equation in $\mathbb{R}^{d}\;(d=2,3)$ with Dirichlet boundary
conditions. We showed that both, the strength of an unknown source and shape of the scattering obstacle are uniquely determined by the covariance of total fields. In numerical examples with synthetic data demonstrate that measurements of such randomly excited waves allow for reconstructions 
of similar quality as in classical active inverse obstacle scattering.


Open questions include uniqueness proofs for different boundary conditions, 
e.g., of Neumann, Robin or transmission type. Such settings seem to require 
new ideas. The same holds true for some different geometrical setups. 
In particular, one may ask if the factorization method is applicable, 
either for proving uniqueness or for numerical reconstructions. 
This question has been answered positively concerning the identification 
of shapes of sources in \cite{GR:22}, but to the best of our knowledge the identifiability of shapes of scattering obstacles, which the factorization method was originally designed for, 
remains unclear in the context of passive imaging. 

\appendix
\renewcommand{\theequation}{\Alph{section}.\arabic{equation}}
\section{Hilbert Schmidt, trace class, and covariance operators}
In this section we collect some known facts on Hilbert Schmidt, trace class, and covariance operators that are used throughout the paper. 

\subsection{Hilbert-Schmidt and trace class operators}
\begin{thmdefi}
Let $X$ and $Y$ be Hilbert spaces. 
A compact linear operator $\mathcal{K}\in L(X,Y)$ 
is called a Hilbert-Schmidt operator if its singular values $\sigma_j(\mathcal{K})$ (counted with multiplicity) are square summable. 
The set $\HS(X,Y)$ of all Hilbert-Schmidt operators from $X$ to $Y$ equipped with the norm $\|\mathcal{K}\|_{\HS}:=(\sum_j \sigma_j(\mathcal{K})^2)^{1/2}$ is a Hilbert space (see \cite[Def. (3.8.1) and Thm. 3.8.3]{Simon2015}).

A compact linear operator $\mathcal{K}\in L(X,Y)$ 
is called a Hilbert-Schmidt trace class if its singular values $\sigma_j(\mathcal{K})$ (counted with multiplicity) are summable. 
The set $S_1(X,Y)$ of all trace class operators from $X$ to $Y$, equipped with the norm $\|\mathcal{K}\|_{S_1}:=\sum_j \sigma_j(\mathcal{K})$ 
is a Banach space (see \cite[Thm. VI.20]{Reed1980}). 
\end{thmdefi}

Obviously, $\mathcal{K}\in \HS(X,Y)$ if and only if $\mathcal{K}^*\in \HS(Y,X)$ since the 
singular values of both operators coincide. 
If $X=L^2(D_1)$ and $Y=L^2(D_2)$ for some measurable 
spaces $D_1$, $D_2$ and $\mathcal{K}\in L(X,Y)$, 
then $\mathcal{K}\in \HS(X,Y)$ if and only if there 
exists a kernel $K\in L^2(D_2\times D_1)$ such that 
$(\mathcal{K}v)(x) = \int_{D_2}K(x,y) v(y)\dy$ for all $v\in L^2(D_2)$ and almost all $x\in D_1$. 
In this case (see \cite[Thm. VI.23]{Reed1980}), 
\be\label{eq:HSL2}
\|\mathcal{K}\|_{\HS} = \|K\|_{L^2}. 
\en

\begin{prop}[trace](\cite[Thm. VI.18]{Reed1980})
Let $\mathcal{K}\in S_1(X,X)$, and let $\{e_j:j\in\mathbb{N}\}$ be a complete 
orthonormal system in $X$. Then 
$ 
\operatorname{tr}(\mathcal{K}):=\sum\nolimits_{j}\langle e_j,\mathcal{K}e_j\rangle
$
is finite and independent of the choice of $\{e_j\}$. It is called 
the \emph{trace} of $\mathcal{K}$.
\end{prop}

It can be shown 
that for Hilbert spaces $X,Y$ and $Z$ the implication
\be\label{eq:productHS}
\mathcal{A}\in \HS(X,Y), \mathcal{B}\in \HS(Y,Z) \quad \Rightarrow 
\mathcal{BA}\in S_1(X,Z)
\en
holds true (see \cite[Thm. VI.22 (h)]{Reed1980}). Moreover, 
\be\label{eq:trace_inner_product}
\langle \mathcal{K}_1,\mathcal{K}_2\rangle_{\HS} := \operatorname{tr}(\mathcal{\mathcal{K}}_1^*\mathcal{K}_2)
\en
defines an inner product on $\HS(X,Y)$ which generates the 
norm in $\HS(X,Y)$ (see \cite[Thm. VI.22 (c)]{Reed1980}).

\subsection{Hilbert space processes}
\begin{defi}\label{defi:HS_process}
Let $X$ be a separable Hilbert space over $\mathbb{\mathbb{K}}\in\{\mathbb{R},\mathbb{C}\}$ and let $(\Omega_P,\Sigma,P)$ be a probability space. 
A mapping $Z: X\times \Omega_P\to \mathbb{K}$, 
$(v,\omega)\mapsto \langle Z,v\rangle(\omega)$ is called a 
\emph{Hilbert space process} on $X$ if it is linear in $v$ and 
if $\langle Z,v\rangle=\langle Z,v\rangle(\omega)$ is a $\mathbb{K}$-valued random variable 
for all $v\in X$. We say that $Z$ has finite second moments if 
there exists a constant $C>0$ such that 
$\EV|\langle Z,v\rangle|^2\leq C\|v\|_X^2$
for all $v\in X$.
We say that $Z$ is Gaussian if the random vector 
$(\langle Z,v_1\rangle,\dots, \langle Z,v_n\rangle)^{\top}$
is multivariate Gaussian for all finite subsets 
$\{v_1,\dots,v_n\}\subset X$. Another Hilbert space process 
$\tilde{Z}:X\times \Omega_P\to \mathbb{K}$ is called 
a \emph{strict version} of $Z$ if $\mathbb{P}[\langle Z,v\rangle = \langle \tilde{Z},v\rangle]=1$ for all 
$v\in X$.
\end{defi}

For a Hilbert space process $Z$ on $X$ with finite second moments it is straightforward to show that 
there exist an expectation $\EV(Z)\in X$ and 
a covariance operator $\Cov(Z)\in L(X)$  which 
are uniquely characterized by 
\begin{align*}
\EV\langle Z,v_1\rangle = \langle \EV(Z),v_1\rangle,\qquad 
\Cov(\langle Z,v_1\rangle,\langle Z,v_2\rangle)
=\langle\Cov(Z)v_1,v_2\rangle_X
\end{align*}
for all $v_1,v_2\in X$.
Moreover, $\Cov(Z)$ is Hermitian, positive semidefinite and 
$\|\Cov(Z)\|\leq C$. $Z$ is called \emph{centered} 
if $\EV(Z)=0$. A centered Hilbert space process 
$Z$ is called a \emph{white noise process} if 
$\Cov(Z)=I$. Given an arbitrary 
orthonormal basis $\{e_1,e_2,\dots\}$ of $X$, a Gaussian white noise process $W$ on $X$ can be constructed 
by $\langle W,v\rangle:=\sum_j \pi_j\langle e_j,v\rangle$
with independent random variables $\pi_j\sim \mathcal{N}(0,1)$. 

Note that a Hilbert space valued random variable 
$z$ with finite second moments, i.e.\ a measurable mapping from $z:\Omega_P\to X$ 
such that $\EV \|z\|_X^2<\infty$ can 
be identified with a Hilbert space process 
$\langle Z,v\rangle(\omega) := \langle z(\omega),v\rangle$. We set 
$\EV(z):=\EV(Z)$ and $\Cov(z):=\Cov(Z)$.
If $\EV(z)=0$, we have 
\begin{align}\label{eq:EV_normsq}
\EV \|z\|^2 = \sum\nolimits_j|\EV\langle z,e_j\rangle|^2
= \sum\nolimits_j\langle \Cov(z)e_j,e_j\rangle 
= \operatorname{tr} \Cov(z).
\end{align}
Not every Hilbert space process can be identified with a 
Hilbert space valued random variable with finite second moments as the example of white noise and eq.~\eqref{eq:EV_normsq} show (as $\operatorname{tr} I =\dim X=\infty)$. However, the following proposition, which, e.g., can be inferred from 
\cite[Prop.~2.1.12]{GN:15}, gives a partial converse:
\begin{prop}\label{prop:regularity_white_noise}
Let $X$ and $V$ be separable Hilbert spaces such that 
$V\subset X\subset V'$ forms a Gelfand triple. Let $Z$ be a centered Gaussian Hilbert space process on $X$ such that 
for a dense, countable subset $B\subset \{v\in V:\|v\|_V\leq 1\}$ we have 
\begin{align}\label{eq:finite_square_norm}
\EV\sup\nolimits_{v\in B} |\langle Z,v\rangle|^2 <\infty.
\end{align}
(Here restriction to a countable subset ensures measurability 
of the supremum.) 
Then there exists a strict version $U$ of $Z$ restricted 
to $V$ and a random variable $u$ with values in $V'$ and 
finite second moments such that 
$\langle u(\omega),v\rangle = \langle U,v\rangle(\omega)$ 
for all $v\in V$ and $\omega\in\Omega$. 
\end{prop}
Note that, by interchanging the order of expectation and supremum, 
\eqref{eq:finite_square_norm} implies that $Z|_V$ has finite 
second moments, and that the supremum 
in \eqref{eq:finite_square_norm} equals $\|u\|_{V'}^2$. 

\begin{lem}\label{lem:finite_square_norm}
Eq.~\eqref{eq:finite_square_norm} holds true if the embedding 
$J:V\hookrightarrow X$ is Hilbert-Schmidt or if 
$V=X$ and $\Cov(Z)$ is trace class. 
\end{lem}

\begin{proof}
Let $\{(e_j,f_j,\sigma_j):j\in\mathbb{N}\}$ be a singular 
value decomposition of $J$. Then
\begin{align}\label{eq:aux_finite_square_norm}
|\langle Z,v\rangle|^2 
&= \left|\sum\nolimits_j \langle Z,f_j\rangle \langle v,f_j\rangle_X\right|^2 
\leq \paren{\sum\nolimits_j \sigma_j^2|\langle Z,f_j\rangle |^2}
\paren{\sum\nolimits_j\sigma_j^{-2}|\langle v,f_j\rangle_X|^2} 
\end{align}
for any $v\in  \operatorname{span}\{f_1,f_2,\dots\}$ 
by the Cauchy-Schwarz inequality. Note that 
$\langle v,f_j\rangle_X 
= \langle Jv,f_j\rangle_X
= \langle v,J^*f_j\rangle_V
= \sigma_j\langle v,e_j\rangle_V.
$ 
Let $B$ denote the intersection of the elements of this 
span with rational coefficients and the unit ball in 
$V$. Then we have  
\begin{align*}
\EV\sup_{v\in B} |\langle Z,v\rangle|^2 
\leq \sum\nolimits_j \sigma_j^2\EV  |\langle Z,f_j\rangle|^2
\sup_{v\in B}\paren{\sum\nolimits_j|\langle v,e_j\rangle_V|^2}
\leq C\|J\|_{\HS}^2. 
\end{align*}
If $\Cov(Z)\in S_1(X)$, let $\{f_j:j\in\mathbb{N}$ be an orthonormal basis of eigenvectors of $\Cov(Z)$ with corresponding eigenvalues $\lambda_j$. As $\EV |\langle Z,f_j\rangle|^2
= \langle \Cov(Z) f_j,f_j\rangle = \lambda_j$, we obtain
from \eqref{eq:aux_finite_square_norm} 
with $\sigma_j=1$ that 
\[
\EV\sup_{v\in B} |\langle Z,v\rangle|^2 
\leq \sum\nolimits_j \EV  |\langle Z,f_j\rangle|^2
\sup_{v\in B}\paren{\sum\nolimits_j|\langle v,f_j\rangle_X|^2}
\leq \operatorname{tr} \Cov(Z).
\]
\end{proof}

\begin{lem}\label{lem:cov_fct_op}
Let $u$ be a centered random variable with values in 
$L^2(M)$ for some Borel measure space $M$. 
Then the two-point covariance function 
$c_u(x,y):=\EV(\overline{u(x)}u(y))$, $x,y\in M$ is the integral 
kernel of the covariance operator $\Cov(u)$.
\end{lem}

\begin{proof}
For any $v_1,v_2\in L^2(M)$ we have 
\begin{align*}
\langle \Cov(u)v_1,v_2\rangle 
&= \EV\left(\overline{\langle u,v_1\rangle}\langle u,v_2\rangle \right)
= \EV \left(\int_M \overline{u(x)}\overline{v_1(x)}\dx
\int_M u(y)v_2(y)\dy\right)\\
&=\int_M\int_M \EV\left(\overline{u(x)}u(y)\right)
\overline{v_1(x)}v_2(y)\dx\dy,
\end{align*}
using Fubini's theorem to interchange the order of 
expectation and integration. 
\end{proof}

If $Y$ is another Hilbert space over $\mathbb{K}$ and 
$\mathcal{G}\in L(X,Y)$, then we can define a Hilbert space process 
$\mathcal{G}Z$ with finite second moments on $Y$ by 
$\langle \mathcal{G}Z,v\rangle:=\langle Z,\mathcal{G}^*v\rangle$ for all $v\in Y$. 
One readily checks that 
\begin{align}\label{eq:linOp_RP}
\EV(\mathcal{G}Z) = \mathcal{G}\EV(Z),\qquad \Cov(\mathcal{G}Z) = \mathcal{G}\Cov(Z)\mathcal{G}^*.
\end{align}

\section*{Acknowledgments}

The authors gratefully acknowledge support from the Deutsche Forschungsgemeinschaft (DFG) through CRC 1456 (grant 432680300, project C04).

\bibliographystyle{siamplain}
\bibliography{references} 

\end{document}